\documentclass[a4paper,10pt]{amsart}
\usepackage{fourier}
\usepackage[margin=2.75cm]{geometry}
\usepackage{mathrsfs}
\usepackage{tikz}
  \usetikzlibrary{arrows}
  \usetikzlibrary{decorations.markings}
  \usetikzlibrary{shapes}
  \usetikzlibrary{matrix}
\usepackage{wasysym}

\newcommand{\C}{\mathcal{C}}
\newcommand{\T}{\mathcal{T}}
\newcommand{\susp}{\widehat{\Sigma}}
\newcommand{\ab}{\mathrm{Ab}}
\newcommand{\ceq}{\mathrel{\mathop:}=}
\newcommand{\calF}{\mathcal{F}}
\newcommand{\catC}{\mathscr{C}}
\newcommand{\nang}{\mathscr{N}}
\DeclareMathOperator{\Hom}{Hom}
\DeclareMathOperator{\Ker}{Ker}
\DeclareMathOperator{\Image}{Im}

\theoremstyle{definition}

\theoremstyle{plain}
\newtheorem{theorem}[subsection]{Theorem}
\newtheorem{lemma}[subsection]{Lemma}

\title{The axioms for $n$-angulated categories}
\author{Petter Andreas Bergh}
\author{Marius Thaule}
\address{Department of Mathematical Sciences, NTNU, NO-7491
  Trondheim, Norway}
\email{bergh@math.ntnu.no}
\email{mariusth@math.ntnu.no}

\begin{document}

% Abstract
\begin{abstract}
  We discuss the axioms for an $n$-angulated category, recently
  introduced by Geiss, Keller and Oppermann in \cite{GKO}.  In
  particular, we introduce a higher ``octahedral axiom'', and show
  that it is equivalent to the mapping cone axiom for an $n$-angulated
  category.  For a triangulated category, the mapping cone axiom, our
  octahedral axiom and the classical octahedral axiom are all
  equivalent.
\end{abstract}

\thanks{The second author was supported by the Norwegian Research
  Council grant Topology, project number 185335/V30.}

\subjclass[2010]{18E30}
\keywords{Triangulated categories, $n$-angulated categories, octahedral axiom}

\maketitle

\section{Introduction}
Triangulated categories were introduced independently in algebraic
geometry by Verdier \cite{Verdier1, Verdier2}, based on ideas of
Grothendieck, and in algebraic topology by Puppe \cite{Puppe}.  These
constructions have since played a crucial role in representation
theory, algebraic geometry, commutative algebra, algebraic topology
and other areas of mathematics (and even theoretical physics).
Recently, Geiss, Keller and Oppermann introduced in \cite{GKO} a new
type of categories, called $n$-angulated categories, which generalize
triangulated categories: the classical triangulated categories are the
special case $n = 3$.  These categories appear for instance when
considering certain $(n - 2)$-cluster tilting subcategories of
triangulated categories.  Conversely, certain $n$-angulated
Calabi--Yau categories yield triangulated Calabi--Yau categories of
higher Calabi--Yau dimension.

The four axioms for $n$-angulated categories are generalizations of
the axioms for triangulated categories.  In this paper, we discuss
these axioms, inspired by works of Neeman \cite{N1,N2}.  First, we
show that the first two of the original axioms can be replaced by two
alternative axioms.  One of these alternative axioms requires that the
collection of $n$-angles be closed under so-called weak isomorphisms,
but not under direct sums and summands.  The other axiom requires that
the collection of $n$-angles be closed only under left rotations, but
not right rotations.  Second, we discuss the axioms that enable us to
complete certain diagrams to morphisms of $n$-angles.  The last of
these axioms says that we can complete diagrams to morphisms of
$n$-angles in such a way that the mapping cone is itself an $n$-angle.
For triangulated categories (that is, when $n = 3$), this axiom is
equivalent to the octahedral axiom, which was one of Verdier's
original axioms.  We show that this generalizes to $n$-angulated
categories.  Namely, we introduce a higher ``octahedral axiom'' for
$n$-angulated categories, and show that this is equivalent to the
mapping cone axiom.  For $n = 3$, that is, for triangulated
categories, our new axiom is almost the same as the classical
octahedral axiom. In fact, it is apparently a bit weaker, but we show
that they are equivalent.  Therefore, for a triangulated category, the
mapping cone axiom, our octahedral axiom and the classical octahedral
axiom are all equivalent.

This paper is organized as follows.  In Section \ref{sec:axioms}, we
recall the definition of $n$-angulated categories from \cite{GKO}.
Then, in Section \ref{sec:axiomsN1N2}, we discuss the first two
axioms, and in Section \ref{sec:axiomN4} we introduce the higher
octahedral axiom.  Finally, in Section \ref{sec:example}, we look at
an example, namely the $n$-angulated categories originating from $(n -
2)$-cluster tilting subcategories of triangulated categories.  We
verify the higher octahedral axiom for these categories, in the case
when $n = 4$.

\section{The Axioms for $n$-Angulated Categories}
\label{sec:axioms}
Throughout this paper, we fix an additive category $\C$ with an
automorphism $\Sigma \colon \C \to \C$, and an integer $n$ greater
than or equal to three. In this section, we recall the set of axioms
for $n$-angulated categories as described in \cite{GKO}.

A sequence of objects and morphisms in $\C$ of the form
\begin{equation*}
  A_1 \xrightarrow{\alpha_1} A_2 \xrightarrow{\alpha_2} \cdots
  \xrightarrow{\alpha_{n - 1}} A_n \xrightarrow{\alpha_n} \Sigma A_1 
\end{equation*}
is called an \emph{$n$-$\Sigma$-sequence}; we shall frequently denote
such sequences by $A_\bullet, B_\bullet$ etc.  The
$n$-$\Sigma$-sequence $A_\bullet$ is \emph{exact} if the induced
sequence
\begin{equation*}
  \cdots \to \Hom_{\C}(B,A_1) \xrightarrow{(\alpha_1)_*}
  \Hom_{\C}(B,A_2) \xrightarrow{(\alpha_2)_*} \cdots
  \xrightarrow{(\alpha_{n-1})_*} \Hom_{\C}(B,A_n)
  \xrightarrow{(\alpha_n)_*} \Hom_{\C}(B, \Sigma A_1) \to \cdots
\end{equation*}
of abelian groups is exact for every object $B \in \C$.  The left and
right \emph{rotations} of $A_{\bullet}$ are the two
$n$-$\Sigma$-sequences
\begin{equation*}
  A_2 \xrightarrow{\alpha_2} A_3 \xrightarrow{\alpha_3} \cdots
  \xrightarrow{\alpha_n} \Sigma A_1 \xrightarrow{(-1)^n
    \Sigma\alpha_1} \Sigma A_2
\end{equation*}
and
\begin{equation*}
  \Sigma^{-1}A_n \xrightarrow{(-1)^n \Sigma^{-1} \alpha_n} A_1
  \xrightarrow{\alpha_1} \cdots \xrightarrow{\alpha_{n - 2}} A_{n-1}
  \xrightarrow{\alpha_{n - 1}} A_n
\end{equation*} 
respectively, and a \emph{trivial} $n$-$\Sigma$-sequence is a sequence
of the form
\begin{equation*}
  A \xrightarrow{1} A \to 0 \to \cdots \to 0 \to \Sigma A
\end{equation*}
or any of its rotations. 

A \emph{morphism} $A_{\bullet} \xrightarrow{\varphi} B_{\bullet}$ of
$n$-$\Sigma$-sequences is a sequence $\varphi =
(\varphi_1,\varphi_2,\ldots,\varphi_n)$ of morphisms in $\C$ such that
the diagram
\begin{center}
  \begin{tikzpicture}[text centered]
    % Top row
    \node (X1) at (0,1.5){$A_1$};
    \node (X2) at (1.5,1.5){$A_2$};
    \node (X3) at (3,1.5){$A_3$};
    \node (Xdots) at (4.5,1.5){$\cdots$};
    \node (Xn) at (6,1.5){$A_n$};
    \node (SX1) at (7.5,1.5){$\Sigma A_1$};

    % Bottom row
    \node (Y1) at (0,0){$B_1$};
    \node (Y2) at (1.5,0){$B_2$};
    \node (Y3) at (3,0){$B_3$};
    \node (Ydots) at (4.5,0){$\cdots$};
    \node (Yn) at (6,0){$B_n$};
    \node (SY1) at (7.5,0){$\Sigma B_1$};

    \begin{scope}[->,font=\scriptsize,midway]
      % Vertical arrows
      \draw (X1) -- node[right]{$\varphi_1$} (Y1);
      \draw (X2) -- node[right]{$\varphi_2$} (Y2);
      \draw (X3) -- node[right]{$\varphi_3$} (Y3);
      \draw (Xn) -- node[right]{$\varphi_n$} (Yn);
      \draw (SX1) -- node[right]{$\Sigma \varphi_1$} (SY1);
      
      % Horizontal arrows
      \draw (X1) -- node[above]{$\alpha_1$} (X2);
      \draw (X2) -- node[above]{$\alpha_2$} (X3);
      \draw (X3) -- node[above]{$\alpha_3$} (Xdots);
      \draw (Xdots) -- node[above]{$\alpha_{n - 1}$} (Xn);
      \draw (Xn) -- node[above]{$\alpha_n$} (SX1);
      \draw (Y1) -- node[above]{$\beta_1$} (Y2);
      \draw (Y2) -- node[above]{$\beta_2$} (Y3);
      \draw (Y3) -- node[above]{$\beta_3$} (Ydots);
      \draw (Ydots) -- node[above]{$\beta_{n - 1}$} (Yn);
      \draw (Yn) -- node[above]{$\beta_n$} (SY1);
    \end{scope}
  \end{tikzpicture}
\end{center}
commutes.  It is an \emph{isomorphism} if
$\varphi_1,\varphi_2,\ldots,\varphi_n$ are all isomorphisms in $\C$,
and a \emph{weak isomorphism} if $\varphi_i$ and $\varphi_{i + 1}$ are
isomorphisms for some $1 \leq i \leq n$ (with $\varphi_{n + 1} \ceq
\Sigma \varphi_1$).  Note that the composition of two weak
isomorphisms need not be a weak isomorphism.  Also, note that if two
$n$-$\Sigma$-sequences $A_{\bullet}$ and $B_{\bullet}$ are weakly
isomorphic through a weak isomorphism $A_{\bullet}
\xrightarrow{\varphi} B_{\bullet}$, then there does not necessarily
exist a weak isomorphism $B_{\bullet} \to A_{\bullet}$ in the opposite
direction.

The category $\C$ is \emph{pre-$n$-angulated} if there exists a
collection $\nang$ of $n$-$\Sigma$-sequences satisfying the following
three axioms:

\begin{itemize}
\item[{\textbf{(N1)}}]
  \begin{itemize}
  \item[(a)] $\nang$ is closed under direct sums, direct summands and
    isomorphisms of $n$-$\Sigma$-sequences.
  \item[(b)] For all $A \in \C$, the trivial $n$-$\Sigma$-sequence
   \begin{equation*}
     A \xrightarrow{1} A \to 0 \to \cdots \to 0 \to \Sigma A
   \end{equation*}
   belongs to $\nang$.
 \item[(c)] For each morphism $\alpha \colon A_1 \to A_2$ in $\C$,
   there exists an $n$-$\Sigma$-sequence in $\nang$ whose first
   morphism is $\alpha$.
 \end{itemize}
\item[{\textbf{(N2)}}] An $n$-$\Sigma$-sequence belongs to $\nang$ if
  and only if its left rotation belongs to $\nang$.
\item[{\textbf{(N3)}}] Each commutative diagram
  \begin{center}
    \begin{tikzpicture}[text centered]
      % Top row
      \node (X1) at (0,1.5){$A_1$};
      \node (X2) at (1.5,1.5){$A_2$};
      \node (X3) at (3,1.5){$A_3$};
      \node (Xdots) at (4.5,1.5){$\cdots$};
      \node (Xn) at (6,1.5){$A_n$};
      \node (SX1) at (7.5,1.5){$\Sigma A_1$};

      % Bottom row
      \node (Y1) at (0,0){$B_1$};
      \node (Y2) at (1.5,0){$B_2$};
      \node (Y3) at (3,0){$B_3$};
      \node (Ydots) at (4.5,0){$\cdots$};
      \node (Yn) at (6,0){$B_n$};
      \node (SY1) at (7.5,0){$\Sigma B_1$};
      
      \begin{scope}[font=\scriptsize,->,midway]
        % Vertical arrows
        \draw (X1) -- node[right]{$\varphi_1$} (Y1);
        \draw (X2) -- node[right]{$\varphi_2$} (Y2);
        \draw[dashed] (X3) -- node[right]{$\varphi_3$} (Y3);
        \draw[dashed] (Xn) -- node[right]{$\varphi_n$} (Yn);
        \draw (SX1) -- node[right]{$\Sigma \varphi_1$} (SY1);
        
        % Horizontal arrows
        \draw (X1) -- node[above]{$\alpha_1$} (X2);
        \draw (X2) -- node[above]{$\alpha_2$} (X3);
        \draw (X3) -- node[above]{$\alpha_3$} (Xdots);
        \draw (Xdots) -- node[above]{$\alpha_{n - 1}$} (Xn);
        \draw (Xn) -- node[above]{$\alpha_n$} (SX1);
        \draw (Y1) -- node[above]{$\beta_1$} (Y2);
        \draw (Y2) -- node[above]{$\beta_2$} (Y3);
        \draw (Y3) -- node[above]{$\beta_3$} (Ydots);
        \draw (Ydots) -- node[above]{$\beta_{n - 1}$} (Yn);
        \draw (Yn) -- node[above]{$\beta_n$} (SY1);
      \end{scope}
    \end{tikzpicture}
  \end{center}
  with rows in $\nang$ can be completed to a morphism of
  $n$-$\Sigma$-sequences.
\end{itemize}

In this case, the collection $\nang$ is a \emph{pre-$n$-angulation} of
the category $\C$ (relative to the automorphism $\Sigma$), and the
$n$-$\Sigma$-sequences in $\nang$ are \emph{$n$-angles}.  If, in
addition, the collection $\nang$ satisfies the following axiom, then
it is an \emph{$n$-angulation} of $\C$, and the category is
\emph{$n$-angulated}:

\begin{itemize}
\item[{\textbf{(N4)}}] In the situation of (N3), the morphisms
  $\varphi_3,\varphi_4,\ldots,\varphi_n$ can be chosen such that the
  mapping cone
  \begin{equation*}
    A_2 \oplus B_1 \xrightarrow{\left[
        \begin{smallmatrix}
          -\alpha_2 & 0\\
          \hfill \varphi_2 & \beta_1
        \end{smallmatrix}
      \right]} A_3 \oplus B_2 \xrightarrow{\left[
        \begin{smallmatrix}
          -\alpha_3 & 0\\
          \hfill \varphi_3 & \beta_2
        \end{smallmatrix}
      \right]} \cdots \xrightarrow{\left[
        \begin{smallmatrix} 
          -\alpha_n & 0\\
          \hfill \varphi_n & \beta_{n - 1}
        \end{smallmatrix}
      \right]} \Sigma A_1 \oplus B_n \xrightarrow{\left[
        \begin{smallmatrix}
          -\Sigma \alpha_1 & 0\\
          \hfill \Sigma \varphi_1 & \beta_n
        \end{smallmatrix}
      \right]} \Sigma A_2 \oplus \Sigma B_1
  \end{equation*}
  belongs to $\nang$.
\end{itemize}

Note that in \cite{GKO}, it was not explicitly assumed that $\nang$ be
closed under isomorphisms, but it follows implicitly from closure
under direct sums.  Since closure under isomorphisms is a crucial part
of many of our proofs, we have included it as a part of axiom (N1)(a).
Note also that by \cite[Proposition 1.5]{GKO}, every $n$-angle in a
pre-$n$-angulated category is exact.  Consequently, the composition of
two consecutive morphisms in an $n$-angle is zero. Finally, note the
similarity with Balmer's recent definition (cf.\ \cite{Balmer}) of
triangulated categories \emph{of order $n$}.

\section{Axioms (N1) and (N2)}
\label{sec:axiomsN1N2}
In this section, we discuss the first two defining axioms (N1) and
(N2) for pre-$n$-angulated categories.  It turns out that we may
replace these axioms by the following ones:
\begin{itemize}
\item[{\textbf{(N1*)}}]
  \begin{itemize}
  \item[(a)] If $A_\bullet \xrightarrow{\varphi} B_\bullet$ is a weak
    isomorphism of exact $n$-$\Sigma$-sequences with $A_\bullet \in
    \nang$, then $B_\bullet$ belongs to $\nang$.
  \item[(b)] For all $A \in \C$, the trivial $n$-$\Sigma$-sequence
    \begin{equation*}
      A \xrightarrow{1} A \to 0 \to \cdots \to 0 \to \Sigma A
    \end{equation*}
    belongs to $\nang$.
  \item[(c)] For each morphism $\alpha \colon A_1 \to A_2$ in $\C$,
    there exists an $n$-$\Sigma$-sequence in $\nang$ whose first
    morphism is $\alpha$.
  \end{itemize}
\item[\textbf{(N2*)}] The left rotation of every $n$-$\Sigma$-sequence
  in $\nang$ also belongs to $\nang$.
\end{itemize}
In axiom (N1*), we do not require that $\nang$ be closed under direct
sums and summands.  However, we do require that $\nang$ be closed
under weak isomorphisms (in one direction), and this is stronger than
requiring that $\nang$ be closed under isomorphisms.  In axiom (N2*),
we only require that $\nang$ be closed under left rotations.  This is
sometimes done when considering triangulated categories, cf.\
\cite{KV}.

Because of the new axiom (N1*)(a), the exact $n$-$\Sigma$-sequences
play an important role in the proofs to come.  We therefore need to
determine which properties a collection $\nang$ of
$n$-$\Sigma$-sequences must satisfy in order for all its elements to
be exact.  We do this in the following result.

\begin{lemma}
  \label{lem:exactness}
  If $\nang$ is a collection of $n$-$\Sigma$-sequences satisfying the
  axioms \emph{(N1)(b)}, \emph{(N2*)} and \emph{(N3)}, then all the
  elements in $\nang$ are exact.
\end{lemma}

\begin{proof}
  Let 
  \begin{center}
    \begin{tikzpicture}
      % A_\bullet
      \node[left] (Abullet) at (-1.25,0){$A_\bullet\colon$};

      % A_1, A_2, ... , A_n, SA_1
      \node (A1) at (0,0){$A_1$};
      \node (A2) at (1.5,0){$A_2$};
      \node (Adots) at (3,0){$\cdots$};
      \node (An) at (4.5,0){$A_n$};
      \node (SA1) at (6,0){$\Sigma A_1$};

      \begin{scope}[font=\scriptsize,->,midway,above]
        % alpha_1, alpha_2, ... , alpha_n
        \draw (A1) -- node{$\alpha_1$} (A2);
        \draw (A2) -- node{$\alpha_2$} (Adots);
        \draw (Adots) -- node{$\alpha_{n - 1}$} (An);
        \draw (An) -- node{$\alpha_n$} (SA1);
      \end{scope}
    \end{tikzpicture}
  \end{center}
  be an $n$-$\Sigma$-sequence in $\nang$, and pick an integer $1 \leq
  j \leq n$.  In the diagram
  \begin{center}
    \begin{tikzpicture}[text centered]
      % Top row
      \node (X1) at (0,1.5){$A_j$};
      \node (X2) at (1.5,1.5){$A_j$};
      \node (X3) at (3.25,1.5){$0$};
      \node (Xdots) at (4.75,1.5){$\cdots$};
      \node (Xn) at (7.25,1.5){$0$};
      \node (SX1) at (9.75,1.5){$\Sigma A_j$};

      % Bottom row
      \node (Y1) at (0,0){$A_j$};
      \node (Y2) at (1.5,0){$A_{j+1}$};
      \node (Y3) at (3.25,0){$A_{j+2}$};
      \node (Ydots) at (4.75,0){$\cdots$};
      \node (Yn) at (7.25,0){$\Sigma A_{j-1}$};
      \node (SY1) at (9.75,0){$\Sigma A_j$};
      
      \begin{scope}[font=\scriptsize,->,midway]
        % Vertical arrows
        \draw[-,double equal sign distance] (X1) -- (Y1);
        \draw (X2) -- node[right]{$\alpha_j$} (Y2);
        \draw[dashed] (X3) -- (Y3);
        \draw[dashed] (Xn) -- (Yn);
        \draw[-,double equal sign distance] (SX1) -- (SY1);
        
        % Horizontal arrows
        \draw (X1) -- node[above]{$1$}(X2);
        \draw (X2) -- (X3);
        \draw (X3) -- (Xdots);
        \draw (Xdots) -- (Xn);
        \draw (Xn) -- (SX1);
        \draw (Y1) -- node[above]{$\alpha_j$} (Y2);
        \draw (Y2) -- node[above]{$\alpha_{j + 1}$} (Y3);
        \draw (Y3) -- node[above]{$\alpha_{j + 2}$} (Ydots);
        \draw (Ydots) -- node[above]{$(-1)^n \Sigma \alpha_{j - 2}$}
          (Yn);
        \draw (Yn) -- node[above]{$(-1)^n \Sigma \alpha_{j - 1}$}
          (SY1);
      \end{scope}
    \end{tikzpicture}
  \end{center}
  the two rows both belong to $\nang$: the top row by (N1)(b), and the
  bottom row by (repeated use of) (N2*).  If $j = 1$ or $j = 2$, the
  two rightmost morphisms in the bottom row have different labels, and
  if $j = n$ then $\alpha_{n+1} = (-1)^n \Sigma \alpha_1$.  By (N3),
  we can complete the diagram to a morphism of $n$-$\Sigma$-sequences,
  hence
  \begin{equation*}
    \alpha_2 \circ \alpha_1 = \alpha_3 \circ \alpha_2 = \cdots =
    \alpha_n \circ \alpha_{n-1} = ( \Sigma \alpha_1 ) \circ \alpha_n =
    0.
  \end{equation*}

  For objects $X,Y \in \C$, denote the abelian group $\Hom_{\C}(X,Y)$
  by $(X,Y)$.  Since all the possible compositions of morphisms from
  $A_\bullet$ are zero, the doubly infinite sequence
  \begin{equation*}
    \cdots \to (B, \Sigma^{i-1} A_n) \xrightarrow{(\Sigma^{i-1}
      \alpha_n)_*} (B, \Sigma^{i} A_1) \xrightarrow{(\Sigma^{i}
      \alpha_1)_*} \cdots \xrightarrow{(\Sigma^{i} \alpha_{n-1})_*}
    (B, \Sigma^i A_n) \xrightarrow{(\Sigma^{i} \alpha_{n})_*} (B,
    \Sigma^{i+1} A_1) \to \cdots
  \end{equation*}
  of abelian groups and maps is a complex for every object $B \in \C$.
  Now pick an integer $1 \leq i \leq n$, and let $f$ be an element in
  $\Ker(\Sigma^i \alpha_j)_*$.  Then $f$ is a morphism in
  $\Hom_{\C}(B, \Sigma^i A_j)$ with $(\Sigma^i \alpha_j)\circ f = 0$.
  Applying the automorphism $\Sigma^{-i}$, we obtain $\alpha_j \circ
  (\Sigma^{-i} f) = 0$, where $\Sigma^{-i} f$ is a morphism in
  $\Hom_{\C}(\Sigma^{-i}B, A_j)$.  Now consider the diagram
  \begin{center}
    \begin{tikzpicture}[text centered]
      % Top row
      \node (X1) at (0,1.5){$\Sigma^{-i}B$};
      \node (X2) at (1.5,1.5){$0$};
      \node (Xdots) at (3,1.5){$\cdots$};
      \node (Xn) at (5.5,1.5){$\Sigma^{1-i}B$};
      \node (SX1) at (8,1.5){$\Sigma^{1-i}B$};

      % Bottom row
      \node (Y1) at (0,0){$A_j$};
      \node (Y2) at (1.5,0){$A_{j+1}$};
      \node (Ydots) at (3,0){$\cdots$};
      \node (Yn) at (5.5,0){$\Sigma A_{j-1}$};
      \node (SY1) at (8,0){$\Sigma A_j$};
      
      \begin{scope}[font=\scriptsize,->,midway]
        % Vertical arrows
        \draw (X1) -- node[right]{$\Sigma^{-i} f$} (Y1);
        \draw (X2) -- (Y2);
        \draw[dashed] (Xn) -- node[right]{$g$} (Yn);
        \draw (SX1) -- node[right]{$\Sigma^{1-i} f$} (SY1);
        
        % Horizontal arrows
        \draw (X1) -- (X2);
        \draw (X2) -- (Xdots);
        \draw (Xdots) -- (Xn);
        \draw (Xn) -- node[above]{$(-1)^n$} (SX1);
        \draw (Y1) -- node[above]{$\alpha_j$} (Y2);
        \draw (Y2) -- node[above]{$\alpha_{j + 1}$} (Ydots);
        \draw (Ydots) -- node[above]{$(-1)^n \Sigma \alpha_{j -
            2}$}(Yn);
        \draw (Yn) -- node[above]{$(-1)^n \Sigma \alpha_{j - 1}$}
          (SY1);
      \end{scope}
    \end{tikzpicture}
  \end{center}
  in which the two rows belong to $\nang$ by (N1)(b) and (repeated use
  of) (N2*).  By (N3), we can complete this diagram to a morphism of
  $n$-$\Sigma$-sequences, and in particular we obtain a morphism $g
  \in \Hom_{\C}( \Sigma^{1-i} B, \Sigma A_{j-1} )$ with
  \begin{equation*}
    (\Sigma \alpha_{j-1}) \circ g = \Sigma^{1-i} f.
  \end{equation*}
  Applying the automorphism $\Sigma^{i-1}$ gives
  \begin{equation*}
    f = (\Sigma^i \alpha_{j-1}) \circ (\Sigma^{i-1} g),
  \end{equation*}
  hence $f \in \Image(\Sigma^i \alpha_{j-1})_*$.  This shows that the
  complex is exact, and so $A_\bullet$ is an exact
  $n$-$\Sigma$-sequence.
\end{proof}

We may now prove that axiom (N1) can be replaced with axiom (N1*).

\begin{theorem}
  \label{tim:newN1}
  If $\nang$ is a collection of $n$-$\Sigma$-sequences satisfying the
  axioms \emph{(N2)} and \emph{(N3)}, then the following are
  equivalent:
  \begin{itemize}
  \item[(1)] $\nang$ satisfies \emph{(N1)},
  \item[(2)] $\nang$ satisfies \emph{(N1*)}.
  \end{itemize}
\end{theorem}

\begin{proof}
  The implication (1) $\Rightarrow$ (2) is part of \cite[Lemma
  1.4]{GKO}, hence we must prove that (1) follows from (2), i.e.\ that
  $\nang$ satisfies (N1)(a) whenever it satisfies (N1*).  Suppose
  therefore that $\nang$ satisfies (N1*).

  Since the collection $\nang$ satisfies the axioms (N1*)(b), (N2) and
  (N3), the $n$-$\Sigma$-sequences in $\nang$ are exact by Lemma
  \ref{lem:exactness}.  Now let $A_{\bullet}$ and $B_{\bullet}$ be
  isomorphic $n$-$\Sigma$-sequences, with $A_{\bullet}$ in $\nang$.
  Then $A_{\bullet}$ is exact, and so $B_{\bullet}$ must also be exact
  since it is isomorphic to $A_{\bullet}$.  Since $A_{\bullet}$ and
  $B_{\bullet}$ are trivially weakly isomorphic through an isomorphism
  $A_{\bullet} \to B_{\bullet}$, the $n$-$\Sigma$-sequence
  $B_{\bullet}$ also belongs to $\nang$.  This shows that $\nang$ is
  closed under isomorphisms.

  Next, we show that $\nang$ is closed under direct sums.  Given two
  $n$-$\Sigma$-sequences
  \begin{center}
    \begin{tikzpicture}
      % A_\bullet, B_\bullet
      \node[left] (Abullet) at (-1.25,0.6){$A_\bullet\colon$};
      \node[left] (Bbullet) at (-1.25,0){$B_\bullet\colon$};

      % A_1, A_2, ... , A_n, SA_1
      \node (A1) at (0,0.6){$A_1$};
      \node (A2) at (1.5,0.6){$A_2$};
      \node (Adots) at (3,0.6){$\cdots$};
      \node (An) at (4.5,0.6){$A_n$};
      \node (SA1) at (6,0.6){$\Sigma A_1$};

      % B_1, B_2, ... , B_n, SB_1
      \node (B1) at (0,0){$B_1$};
      \node (B2) at (1.5,0){$B_2$};
      \node (Bdots) at (3,0){$\cdots$};
      \node (Bn) at (4.5,0){$B_n$};
      \node (SB1) at (6,0){$\Sigma B_1$};

      \begin{scope}[font=\scriptsize,->,midway,above]
        % alpha_1, alpha_2, ... , alpha_n
        \draw (A1) -- node{$\alpha_1$} (A2);
        \draw (A2) -- node{$\alpha_2$} (Adots);
        \draw (Adots) -- node{$\alpha_{n - 1}$} (An);
        \draw (An) -- node{$\alpha_n$} (SA1);
        
        % beta_1, beta_2, ... , beta_n
        \draw (B1) -- node{$\beta_1$} (B2);
        \draw (B2) -- node{$\beta_2$} (Bdots);
        \draw (Bdots) -- node{$\beta_{n - 1}$} (Bn);
        \draw (Bn) -- node{$\beta_n$} (SB1);
      \end{scope}
    \end{tikzpicture}
  \end{center}
  in $\nang$, the direct sum $A_\bullet \oplus B_\bullet$ is exact,
  since each of the sequences is exact by the above.  Now use (N1*)(c)
  to complete the first morphism in $A_\bullet \oplus B_\bullet$ to an
  $n$-$\Sigma$-sequence
  \begin{equation*}
    A_1 \oplus B_1 \xrightarrow{\left[
        \begin{smallmatrix}
          \alpha_1 & 0\\
          0 & \beta_1
        \end{smallmatrix}
      \right]} A_2 \oplus B_2 \xrightarrow{\gamma_2} C_3
    \xrightarrow{\gamma_3} \cdots \xrightarrow{\gamma_{n-1}} C_n
    \xrightarrow{\gamma_n} \Sigma A_1 \oplus \Sigma B_1
  \end{equation*}
  in $\nang$.  By (N3), the two commutative diagrams
  \begin{center}
    \begin{tikzpicture}[text centered]
      % Top row
      \node (X1) at (0,1.5){$A_1 \oplus B_1$};
      \node (X2) at (2,1.5){$A_2 \oplus B_2$};
      \node (X3) at (4,1.5){$C_3$};
      \node (Xdots) at (5.5,1.5){$\cdots$};
      \node (Xn) at (7,1.5){$C_n$};
      \node (SX1) at (9,1.5){$\Sigma A_1 \oplus \Sigma B_1$};

      % Bottom row
      \node (Y1) at (0,0){$A_1$};
      \node (Y2) at (2,0){$A_2$};
      \node (Y3) at (4,0){$A_3$};
      \node (Ydots) at (5.5,0){$\cdots$};
      \node (Yn) at (7,0){$A_n$};
      \node (SY1) at (9,0){$\Sigma A_1$};
      
      \begin{scope}[font=\scriptsize,->,midway]
        % Vertical arrows
        \draw (X1) -- node[right]{$\left[
            \begin{smallmatrix}
              1 & 0
            \end{smallmatrix}
          \right]$} (Y1);
        \draw (X2) -- node[right]{$\left[
            \begin{smallmatrix}
              1 & 0
            \end{smallmatrix}
          \right]$} (Y2);
        \draw[dashed] (X3) -- node[right]{$\varphi_3$} (Y3);
        \draw[dashed] (Xn) -- node[right]{$\varphi_n$} (Yn);
        \draw (SX1) -- node[right]{$\left[
            \begin{smallmatrix}
              1 & 0
            \end{smallmatrix}
          \right]$} (SY1);
        
        % Horizontal arrows
        \draw (X1) -- node[above]{$\left[
            \begin{smallmatrix}
              \alpha_1 & 0\\
              0 & \beta_1
            \end{smallmatrix}
          \right]$} (X2);
        \draw (X2) -- node[above]{$\gamma_2$} (X3);
        \draw (X3) -- node[above]{$\gamma_3$} (Xdots);
        \draw (Xdots) -- node[above]{$\gamma_{n - 1}$} (Xn);
        \draw (Xn) -- node[above]{$\gamma_n$} (SX1);
        \draw (Y1) -- node[above]{$\alpha_1$} (Y2);
        \draw (Y2) -- node[above]{$\alpha_2$} (Y3);
        \draw (Y3) -- node[above]{$\alpha_3$} (Ydots);
        \draw (Ydots) -- node[above]{$\alpha_{n - 1}$} (Yn);
        \draw (Yn) -- node[above]{$\alpha_n$} (SY1);
      \end{scope}
    \end{tikzpicture}
  \end{center}
  
  \begin{center}
    \begin{tikzpicture}[text centered]
      % Top row
      \node (X1) at (0,1.5){$A_1 \oplus B_1$};
      \node (X2) at (2,1.5){$A_2 \oplus B_2$};
      \node (X3) at (4,1.5){$C_3$};
      \node (Xdots) at (5.5,1.5){$\cdots$};
      \node (Xn) at (7,1.5){$C_n$};
      \node (SX1) at (9,1.5){$\Sigma A_1 \oplus \Sigma B_1$};

      % Bottom row
      \node (Y1) at (0,0){$B_1$};
      \node (Y2) at (2,0){$B_2$};
      \node (Y3) at (4,0){$B_3$};
      \node (Ydots) at (5.5,0){$\cdots$};
      \node (Yn) at (7,0){$B_n$};
      \node (SY1) at (9,0){$\Sigma B_1$};
      
      \begin{scope}[font=\scriptsize,->,midway]
        % Vertical arrows
        \draw (X1) -- node[right]{$\left[
            \begin{smallmatrix}
              0 & 1
            \end{smallmatrix}
          \right]$} (Y1);
        \draw (X2) -- node[right]{$\left[
            \begin{smallmatrix}
              0 & 1
            \end{smallmatrix}
          \right]$} (Y2);
        \draw[dashed] (X3) -- node[right]{$\psi_3$} (Y3);
        \draw[dashed] (Xn) -- node[right]{$\psi_n$} (Yn);
        \draw (SX1) -- node[right]{$\left[
            \begin{smallmatrix}
              0 & 1
            \end{smallmatrix}
          \right]$} (SY1);
        
        % Horizontal arrows
        \draw (X1) -- node[above]{$\left[
            \begin{smallmatrix}
              \alpha_1 & 0\\
              0 & \beta_1
            \end{smallmatrix}
          \right]$} (X2);
        \draw (X2) -- node[above]{$\gamma_2$} (X3);
        \draw (X3) -- node[above]{$\gamma_3$} (Xdots);
        \draw (Xdots) -- node[above]{$\gamma_{n - 1}$} (Xn);
        \draw (Xn) -- node[above]{$\gamma_n$} (SX1);
        \draw (Y1) -- node[above]{$\beta_1$} (Y2);
        \draw (Y2) -- node[above]{$\beta_2$} (Y3);
        \draw (Y3) -- node[above]{$\beta_3$} (Ydots);
        \draw (Ydots) -- node[above]{$\beta_{n - 1}$} (Yn);
        \draw (Yn) -- node[above]{$\beta_n$} (SY1);
      \end{scope}
    \end{tikzpicture}
  \end{center}
  can be completed to morphisms of $n$-$\Sigma$-sequences, since the
  sequences involved are all in $\nang$.  This gives a weak
  isomorphism
  \begin{center}
    \begin{tikzpicture}[text centered]
      % Top row
      \node (X1) at (0,1.5){$A_1 \oplus B_1$};
      \node (X2) at (2.25,1.5){$A_2 \oplus B_2$};
      \node (X3) at (4.5,1.5){$C_3$};
      \node (Xdots) at (6.5,1.5){$\cdots$};
      \node (Xn) at (9,1.5){$C_n$};
      \node (SX1) at (11.5,1.5){$\Sigma A_1 \oplus \Sigma B_1$};

      % Bottom row
      \node (Y1) at (0,0){$A_1 \oplus B_1$};
      \node (Y2) at (2.25,0){$A_2 \oplus B_2$};
      \node (Y3) at (4.5,0){$A_3 \oplus B_3$};
      \node (Ydots) at (6.5,0){$\cdots$};
      \node (Yn) at (9,0){$A_n \oplus B_n$};
      \node (SY1) at (11.5,0){$\Sigma A_1 \oplus \Sigma B_1$};
      
      \begin{scope}[font=\scriptsize,->,midway]
        % Vertical arrows
        \draw[-,double equal sign distance] (X1) -- (Y1);
        \draw[-,double equal sign distance] (X2) -- (Y2);
        \draw[dashed] (X3) -- node[right]{$\left[
            \begin{smallmatrix}
              \varphi_3\\
              \psi_3
            \end{smallmatrix}
          \right]$} (Y3);
        \draw[dashed] (Xn) -- node[right]{$\left[
            \begin{smallmatrix}
              \varphi_n \\
              \psi_n
            \end{smallmatrix}
          \right]$} (Yn);
        \draw[-,double equal sign distance] (SX1) -- (SY1);
        
        % Horizontal arrows
        \draw (X1) -- node[above]{$\left[
            \begin{smallmatrix}
              \alpha_1 & 0\\
              0 & \beta_1
            \end{smallmatrix}
          \right]$} (X2);
        \draw (X2) -- node[above]{$\gamma_2$} (X3);
        \draw (X3) -- node[above]{$\gamma_3$} (Xdots);
        \draw (Xdots) -- node[above]{$\gamma_{n - 1}$} (Xn);
        \draw (Xn) -- node[above]{$\gamma_n$} (SX1);
        \draw (Y1) -- node[above]{$\left[
            \begin{smallmatrix}
              \alpha_1 & 0\\
              0 & \beta_1
            \end{smallmatrix}
          \right]$} (Y2);
        \draw (Y2) -- node[above]{$\left[
            \begin{smallmatrix}
              \alpha_2 & 0\\
              0 & \beta_2
            \end{smallmatrix}
          \right]$} (Y3);
        \draw (Y3) -- node[above]{$\left[
            \begin{smallmatrix}
              \alpha_3 & 0\\
              0 & \beta_3
            \end{smallmatrix}
          \right]$} (Ydots);
        \draw (Ydots) -- node[above]{$\left[
            \begin{smallmatrix}
              \alpha_{n-1} & 0\\
              0 & \beta_{n-1}
            \end{smallmatrix}
          \right]$} (Yn);
        \draw (Yn) -- node[above]{$\left[
            \begin{smallmatrix}
              \alpha_n & 0\\
              0 & \beta_n
            \end{smallmatrix}
          \right]$} (SY1);
      \end{scope}
    \end{tikzpicture}
  \end{center}
  of $n$-$\Sigma$-sequences.  The top sequence belongs to $\nang$ and
  is therefore exact, whereas the bottom sequence $A_\bullet \oplus
  B_\bullet$ is also exact.  From (N1*)(a) we conclude that $A_\bullet
  \oplus B_\bullet$ belongs to $\nang$.

  Finally, we show that $\nang$ is closed under direct summands.
  Suppose therefore that $A_\bullet$ and $B_\bullet$ are
  $n$-$\Sigma$-sequences as above, that $B_\bullet$ belongs to $\nang$
  (hence $B_\bullet$ is exact), and that $A_\bullet$ is a direct
  summand of $B_\bullet$.  Then there exists a diagram
  \begin{center}
    \begin{tikzpicture}[text centered]
      % Top row
      \node (X1) at (0,1.5){$A_1$};
      \node (X2) at (1.5,1.5){$A_2$};
      \node (X3) at (3,1.5){$A_3$};
      \node (Xdots) at (4.5,1.5){$\cdots$};
      \node (Xn) at (6,1.5){$A_n$};
      \node (SX1) at (7.5,1.5){$\Sigma A_1$};

      % Middle row
      \node (Y1) at (0,0){$B_1$};
      \node (Y2) at (1.5,0){$B_2$};
      \node (Y3) at (3,0){$B_3$};
      \node (Ydots) at (4.5,0){$\cdots$};
      \node (Yn) at (6,0){$B_n$};
      \node (SY1) at (7.5,0){$\Sigma B_1$};
      
      % Bottom row
      \node (Z1) at (0,-1.5){$A_1$};
      \node (Z2) at (1.5,-1.5){$A_2$};
      \node (Z3) at (3,-1.5){$A_3$};
      \node (Zdots) at (4.5,-1.5){$\cdots$};
      \node (Zn) at (6,-1.5){$A_n$};
      \node (SZ1) at (7.5,-1.5){$\Sigma A_1$};

      \begin{scope}[font=\scriptsize,->,midway]
        % Vertical arrows
        \draw (X1) -- node[right]{$\varphi_1$} (Y1);
        \draw (X2) -- node[right]{$\varphi_2$} (Y2);
        \draw (X3) -- node[right]{$\varphi_3$} (Y3);
        \draw (Xn) -- node[right]{$\varphi_n$} (Yn);
        \draw (SX1) -- node[right]{$\Sigma \varphi_1$} (SY1);
        \draw (Y1) -- node[right]{$\psi_1$} (Z1);
        \draw (Y2) -- node[right]{$\psi_2$} (Z2);
        \draw (Y3) -- node[right]{$\psi_3$} (Z3);
        \draw (Yn) -- node[right]{$\psi_n$} (Zn);
        \draw (SY1) -- node[right]{$\Sigma \psi_1$} (SZ1);

        % Horizontal arrows
        \draw (X1) -- node[above]{$\alpha_1$} (X2);
        \draw (X2) -- node[above]{$\alpha_2$} (X3);
        \draw (X3) -- node[above]{$\alpha_3$} (Xdots);
        \draw (Xdots) -- node[above]{$\alpha_{n - 1}$} (Xn);
        \draw (Xn) -- node[above]{$\alpha_n$} (SX1);
        \draw (Y1) -- node[above]{$\beta_1$} (Y2);
        \draw (Y2) -- node[above]{$\beta_2$} (Y3);
        \draw (Y3) -- node[above]{$\beta_3$} (Ydots);
        \draw (Ydots) -- node[above]{$\beta_{n - 1}$} (Yn);
        \draw (Yn) -- node[above]{$\beta_n$} (SY1);
        \draw (Z1) -- node[above]{$\alpha_1$} (Z2);
        \draw (Z2) -- node[above]{$\alpha_2$} (Z3);
        \draw (Z3) -- node[above]{$\alpha_3$} (Zdots);
        \draw (Zdots) -- node[above]{$\alpha_{n - 1}$} (Zn);
        \draw (Zn) -- node[above]{$\alpha_n$} (SZ1);
      \end{scope}
    \end{tikzpicture}
  \end{center}
  of morphisms $A_{\bullet} \xrightarrow{\varphi} B_{\bullet}$ and
  $B_{\bullet} \xrightarrow{\psi} A_{\bullet}$ of
  $n$-$\Sigma$-sequences, with $\psi_i \circ \varphi_i = 1_{A_i}$ for
  all $i$.  For every object $Z$ in $\C$, the sequence $\Hom_{\C}(Z,
  A_\bullet )$ of abelian groups and maps is a direct summand of the
  exact sequence $\Hom_{\C}(Z, B_\bullet )$, and is therefore itself
  exact.  Consequently, the $n$-$\Sigma$-sequence $A_\bullet$ is
  exact.  Now use (N1*)(c) to complete the first morphism in
  $A_\bullet$ to an $n$-$\Sigma$-sequence
  \begin{center}
    \begin{tikzpicture}
      % D_\bullet
      \node[left] (Dbullet) at (-1.25,0.6){$D_\bullet\colon$};
      
      % A_1, A_2, ... , A_n, SA_1
      \node (A1) at (0,0.6){$A_1$};
      \node (A2) at (1.5,0.6){$A_2$};
      \node (D3) at (3,0.6){$D_3$};
      \node (Ddots) at (4.5,0.6){$\cdots$};
      \node (Dn) at (6,0.6){$D_n$};
      \node (SA1) at (7.5,0.6){$\Sigma A_1$};

      \begin{scope}[font=\scriptsize,->,midway,above]
        % alpha_1, alpha_2, ... , alpha_n
        \draw (A1) -- node{$\alpha_1$} (A2);
        \draw (A2) -- node{$\delta_2$} (D3);
        \draw (D3) -- node{$\delta_3$} (Ddots);
        \draw (Ddots) -- node{$\delta_{n - 1}$} (Dn);
        \draw (Dn) -- node{$\delta_n$} (SA1);
      \end{scope}
    \end{tikzpicture}
  \end{center}
  in $\nang$ (in particular, $D_\bullet$ is exact).  Using this
  sequence, we can obtain a diagram
  \begin{center}
    \begin{tikzpicture}[text centered]
      % Top row
      \node (X1) at (0,1.5){$A_1$};
      \node (X2) at (1.5,1.5){$A_2$};
      \node (X3) at (3,1.5){$D_3$};
      \node (Xdots) at (4.5,1.5){$\cdots$};
      \node (Xn) at (6,1.5){$D_n$};
      \node (SX1) at (7.5,1.5){$\Sigma A_1$};
      
      % Middle row
      \node (Y1) at (0,0){$B_1$};
      \node (Y2) at (1.5,0){$B_2$};
      \node (Y3) at (3,0){$B_3$};
      \node (Ydots) at (4.5,0){$\cdots$};
      \node (Yn) at (6,0){$B_n$};
      \node (SY1) at (7.5,0){$\Sigma B_1$};
      
      % Bottom row
      \node (Z1) at (0,-1.5){$A_1$};
      \node (Z2) at (1.5,-1.5){$A_2$};
      \node (Z3) at (3,-1.5){$A_3$};
      \node (Zdots) at (4.5,-1.5){$\cdots$};
      \node (Zn) at (6,-1.5){$A_n$};
      \node (SZ1) at (7.5,-1.5){$\Sigma A_1$};
    
      \begin{scope}[font=\scriptsize,->,midway]
        % Vertical arrows
        \draw (X1) -- node[right]{$\varphi_1$} (Y1);
        \draw (X2) -- node[right]{$\varphi_2$} (Y2);
        \draw[dashed] (X3) -- node[right]{$\theta_3$} (Y3);
        \draw[dashed] (Xn) -- node[right]{$\theta_n$} (Yn);
        \draw (SX1) -- node[right]{$\Sigma \varphi_1$} (SY1);
        \draw (Y1) -- node[right]{$\psi_1$} (Z1);
        \draw (Y2) -- node[right]{$\psi_2$} (Z2);
        \draw (Y3) -- node[right]{$\psi_3$} (Z3);
        \draw (Yn) -- node[right]{$\psi_n$} (Zn);
        \draw (SY1) -- node[right]{$\Sigma \psi_1$} (SZ1);
      
        % Horizontal arrows
        \draw (X1) -- node[above]{$\alpha_1$} (X2);
        \draw (X2) -- node[above]{$\delta_2$} (X3);
        \draw (X3) -- node[above]{$\delta_3$} (Xdots);
        \draw (Xdots) -- node[above]{$\delta_{n - 1}$} (Xn);
        \draw (Xn) -- node[above]{$\delta_n$} (SX1);
        \draw (Y1) -- node[above]{$\beta_1$} (Y2);
        \draw (Y2) -- node[above]{$\beta_2$} (Y3);
        \draw (Y3) -- node[above]{$\beta_3$} (Ydots);
        \draw (Ydots) -- node[above]{$\beta_{n - 1}$} (Yn);
        \draw (Yn) -- node[above]{$\beta_n$} (SY1);
        \draw (Z1) -- node[above]{$\alpha_1$} (Z2);
        \draw (Z2) -- node[above]{$\alpha_2$} (Z3);
        \draw (Z3) -- node[above]{$\alpha_3$} (Zdots);
        \draw (Zdots) -- node[above]{$\alpha_{n - 1}$} (Zn);
        \draw (Zn) -- node[above]{$\alpha_n$} (SZ1);
      \end{scope}
    \end{tikzpicture}
  \end{center}
  whose rows are $D_\bullet$, $B_\bullet$ and $A_\bullet$.  The top
  half of this diagram is a morphism $D_{\bullet} \xrightarrow{\theta}
  B_{\bullet}$, which we obtain from (N3), whereas the lower half is
  the morphism $B_{\bullet} \xrightarrow{\psi} A_{\bullet}$.
  Moreover, the composition $D_{\bullet} \xrightarrow{\psi \circ
    \theta} A_{\bullet}$ is a weak isomorphism, since $\psi_1 \circ
  \varphi_1 = 1_{A_1}$ and $\psi_2 \circ \varphi_2 = 1_{A_2}$.  Since
  both $D_{\bullet}$ and $A_{\bullet}$ are exact, and $D_{\bullet} \in
  \nang$, the sequence $A_{\bullet}$ belongs to $\nang$ by (N1*)(a).
  This shows that the collection $\nang$ is closed under direct
  summands.  We have now proved that $\nang$ is closed under
  isomorphisms, direct sums and direct summands, which is axiom
  (N1)(a).
\end{proof}

Next, we study the rotation axiom (N2).  The following result shows
that when we replace (N1) with (N1*), then we can also replace (N2)
with the weaker version with (N2*).  In other words, in the rotation
axiom we only need to require that the left rotation of an
$n$-$\Sigma$ sequence in $\nang$ also belongs to $\nang$.

\begin{theorem}
  \label{thm:rightrotation}
  If $\nang$ is a collection of $n$-$\Sigma$-sequences satisfying the
  axioms \emph{(N1*)} and \emph{(N3)}, then the following are
  equivalent:
  \begin{itemize}
  \item[(1)] $\nang$ satisfies \emph{(N2)},
  \item[(2)] $\nang$ satisfies \emph{(N2*)}.
  \end{itemize}
\end{theorem}

\begin{proof}
  The implication (1) $\Rightarrow$ (2) is trivial.  Assume therefore
  that $\nang$ satisfies (N2*), and let
  \begin{center}
    \begin{tikzpicture}
      % A_\bullet
      \node[left] (Abullet) at (-1.25,0.6){$A_\bullet\colon$};

      % A_1, A_2, ... , A_n, SA_1
      \node (A1) at (0,0.6){$A_1$};
      \node (A2) at (1.5,0.6){$A_2$};
      \node (Adots) at (3,0.6){$\cdots$};
      \node (An) at (4.5,0.6){$A_n$};
      \node (SA1) at (6,0.6){$\Sigma A_1$};

      \begin{scope}[font=\scriptsize,->,midway,above]
        % alpha_1, alpha_2, ... , alpha_n
        \draw (A1) -- node{$\alpha_1$} (A2);
        \draw (A2) -- node{$\alpha_2$} (Adots);
        \draw (Adots) -- node{$\alpha_{n - 1}$} (An);
        \draw (An) -- node{$\alpha_n$} (SA1);
      \end{scope}
    \end{tikzpicture}
  \end{center}
  be an $n$-$\Sigma$-sequence in $\nang$.  By repeatedly applying
  (N2*), we obtain the $n$-$\Sigma$-sequence
  \begin{center}
    \begin{tikzpicture}
      % A_1, A_2, ... , A_n, SA_1
      \node (X1) at (0.5,0.6){$A_n$};
      \node (X2) at (2,0.6){$\Sigma A_1$};
      \node (Xdots) at (4,0.6){$\cdots$};
      \node (Xn) at (6.5,0.6){$\Sigma A_{n-1}$};
      \node (SX1) at (9,0.6){$\Sigma A_n$};

      \begin{scope}[font=\scriptsize,->,midway,above]
        % alpha_1, alpha_2, ... , alpha_n
        \draw (X1) -- node{$\alpha_n$} (X2);
        \draw (X2) -- node{$(-1)^n \Sigma \alpha_1$} (Xdots);
        \draw (Xdots) -- node{$(-1)^n \Sigma\alpha_{n-2}$} (Xn);
        \draw (Xn) -- node{$(-1)^n \Sigma \alpha_{n-1}$} (SX1);      
      \end{scope}
    \end{tikzpicture}
  \end{center}
  in $\nang$.  Now use (N1*)(c) to complete the morphism $\Sigma^{-1}
  A_n \xrightarrow{(-1)^n \Sigma^{-1} \alpha_n} A_1$ to an
  $n$-$\Sigma$-sequence
  \begin{center}
    \begin{tikzpicture}
      % A_1, A_2, ... , A_n, SA_1
      \node (X1) at (0,0.6){$\Sigma^{-1} A_n$};
      \node (X2) at (2.5,0.6){$A_1$};
      \node (X3) at (4,0.6){$B_3$};
      \node (Xdots) at (5.5,0.6){$\cdots$};
      \node (Xn) at (7,0.6){$B_n$};
      \node (SX1) at (8.5,0.6){$A_n$};

      \begin{scope}[font=\scriptsize,->,midway,above]      
        % alpha_1, alpha_2, ... , alpha_n
        \draw (X1) -- node{$(-1)^n \Sigma^{-1} \alpha_n$} (X2);
        \draw (X2) -- node{$\beta_2$} (X3);
        \draw (X3) -- node{$\beta_3$} (Xdots);
        \draw (Xdots) -- node{$\beta_{n-1}$} (Xn);
        \draw (Xn) -- node{$\beta_n$} (SX1);        
      \end{scope}
    \end{tikzpicture}
  \end{center}
  in $\nang$.  By repeated use of (N2*), we obtain the
  $n$-$\Sigma$-sequence
  \begin{center}
    \begin{tikzpicture}
      % A_1, A_2, ... , A_n, SA_1
      \node (X1) at (0,0.6){$A_n$};
      \node (X2) at (1.5,0.6){$\Sigma A_1$};
      \node (X3) at (3.5,0.6){$\Sigma B_3$};
      \node (Xdots) at (5.5,0.6){$\cdots$};
      \node (Xn) at (7.5,0.6){$\Sigma B_n$};
      \node (SX1) at (9.5,0.6){$\Sigma A_n$};

      \begin{scope}[font=\scriptsize,->,midway,above]
        % alpha_1, alpha_2, ... , alpha_n
        \draw (X1) -- node{$\alpha_n$} (X2);
        \draw (X2) -- node{$(-1)^n \Sigma \beta_2$} (X3);
        \draw (X3) -- node{$(-1)^n \Sigma \beta_3$} (Xdots);
        \draw (Xdots) -- node{$(-1)^n \Sigma \beta_{n-1}$} (Xn);
        \draw (Xn) -- node{$(-1)^n \Sigma \beta_n$} (SX1);      
      \end{scope}
    \end{tikzpicture}
  \end{center}
  in $\nang$.  By (N3), we may complete the diagram
  \begin{center}
    \begin{tikzpicture}
      % Top row
      \node (X1) at (0,1.5){$A_n$};
      \node (X2) at (1.5,1.5){$\Sigma A_1$};
      \node (X3) at (3.5,1.5){$\Sigma B_3$};
      \node (Xdots) at (5.5,1.5){$\cdots$};
      \node (Xn) at (7.5,1.5){$\Sigma B_n$};
      \node (SX1) at (9.5,1.5){$\Sigma A_n$};
            
      % Bottom row
      \node (Y1) at (0,0){$A_n$};
      \node (Y2) at (1.5,0){$\Sigma A_1$};
      \node (Y3) at (3.5,0){$\Sigma A_2$};
      \node (Ydots) at (5.5,0){$\cdots$};
      \node (Yn) at (7.5,0){$\Sigma A_{n-1}$};
      \node (SY1) at (9.5,0){$\Sigma A_n$};

      \begin{scope}[font=\scriptsize,->,midway]
        % Top horizontal arrows
        \draw (X1) -- node[above]{$\alpha_n$} (X2);
        \draw (X2) -- node[above]{$(-1)^n \Sigma \beta_2$} (X3);
        \draw (X3) -- node[above]{$(-1)^n \Sigma \beta_3$} (Xdots);
        \draw (Xdots) -- node[above]{$(-1)^n \Sigma \beta_{n-1}$}
          (Xn);
        \draw (Xn) -- node[above]{$(-1)^n \Sigma \beta_n$} (SX1);
        
        % Bottom horizontal arrows
        \draw (Y1) -- node[above]{$\alpha_n$} (Y2);
        \draw (Y2) -- node[above]{$(-1)^n \Sigma \alpha_1$} (Y3);
        \draw (Y3) -- node[above]{$(-1)^n \Sigma \alpha_2$} (Ydots);
        \draw (Ydots) -- node[above]{$(-1)^n \Sigma\alpha_{n-2}$}
          (Yn);
        \draw (Yn) -- node[above]{$(-1)^n \Sigma \alpha_{n-1}$}
          (SY1);
       
        % Vertical arrows 
        \draw[-,double equal sign distance] (X1) -- (Y1);
        \draw[-,double equal sign distance] (X2) -- (Y2);
        \draw[dashed] (X3) -- node[right]{$\varphi_3$} (Y3);
        \draw[dashed] (Xn) -- node[right]{$\varphi_n$} (Yn);
        \draw[-,double equal sign distance] (SX1) -- (SY1);  
      \end{scope}
    \end{tikzpicture}
  \end{center}
  and obtain a morphism of $n$-$\Sigma$-sequences.  By applying the
  automorphism $\Sigma^{-1}$ to the whole diagram, and multiplying all
  maps with $(-1)^n$, we obtain a weak isomorphism
  \begin{center}
    \begin{tikzpicture}
      % Top row
      \node (X1) at (0,1.5){$\Sigma^{-1}A_n$};
      \node (X2) at (2.25,1.5){$A_1$};
      \node (X3) at (4.25,1.5){$B_3$};
      \node (Xdots) at (6.25,1.5){$\cdots$};
      \node (Xn) at (8.25,1.5){$B_n$};
      \node (SX1) at (10.25,1.5){$A_n$};
            
      % Bottom row
      \node (Y1) at (0,0){$\Sigma^{-1} A_n$};
      \node (Y2) at (2.25,0){$A_1$};
      \node (Y3) at (4.25,0){$A_2$};
      \node (Ydots) at (6.25,0){$\cdots$};
      \node (Yn) at (8.25,0){$A_{n-1}$};
      \node (SY1) at (10.25,0){$A_n$};

      \begin{scope}[font=\scriptsize,->,midway]
        % Top horizontal maps
        \draw (X1) -- node[above]{$(-1)^n \Sigma^{-1} \alpha_n$} (X2);
        \draw (X2) -- node[above]{$\beta_2$} (X3);
        \draw (X3) -- node[above]{$\beta_3$} (Xdots);
        \draw (Xdots) -- node[above]{$\beta_{n-1}$} (Xn);
        \draw (Xn) -- node[above]{$\beta_n$} (SX1);
        
        % Bottom horizontal maps
        \draw (Y1) -- node[above]{$(-1)^n \Sigma^{-1} \alpha_n$} (Y2);
        \draw (Y2) -- node[above]{$\alpha_1$} (Y3);
        \draw (Y3) -- node[above]{$\alpha_2$} (Ydots);
        \draw (Ydots) -- node[above]{$\alpha_{n-2}$} (Yn);
        \draw (Yn) -- node[above]{$\alpha_{n-1}$} (SY1);
        
        % Vertical arrows   
        \draw (X1) -- node[right]{$(-1)^n$} (Y1);
        \draw (X2) -- node[right]{$(-1)^n$} (Y2);
        \draw (X3) -- node[right]{$(-1)^n \Sigma^{-1} \varphi_3$}
          (Y3);
        \draw (Xn) -- node[right]{$(-1)^n \Sigma^{-1} \varphi_n$}
          (Yn);
        \draw (SX1) -- node[right]{$(-1)^n$} (SY1);    
      \end{scope}
    \end{tikzpicture}
  \end{center}
  of $n$-$\Sigma$-sequences.  The top row belongs to $\nang$ and is
  therefore exact by Lemma \ref{lem:exactness}, whereas the bottom row
  is the right rotation of $A_\bullet$.  Since $A_\bullet$ is exact,
  so is its right rotation, and from (N1*)(a) we conclude that this
  right rotation also belongs to $\nang$.
\end{proof}

Collecting the results in this section gives the following.

\begin{theorem}
  \label{thm:replaceN1andN2}
  For a collection $\nang$ of $n$-$\Sigma$-sequences, the following
  are equivalent:
  \begin{itemize}
  \item[(1)] $\nang$ satisfies \emph{(N1)}, \emph{(N2)} and
    \emph{(N3)},
  \item[(2)] $\nang$ satisfies \emph{(N1*)}, \emph{(N2)} and
    \emph{(N3)},
  \item[(3)] $\nang$ satisfies \emph{(N1*)}, \emph{(N2*)} and
    \emph{(N3)}.
  \end{itemize}
\end{theorem}

\section{Axiom (N4)}
\label{sec:axiomN4}
For triangulated categories, it is a well-known fact that Verdier's
original octahedral axiom has several equivalent representations, see
e.g.\ \cite{HJ} for a discussion.  It is natural to ask whether this
also holds true for general $n$-angulated categories.  We prove in
this section that it does: we introduce a higher ``octahedral axiom''
(N4*) for $n$-angulated categories, and show that it is equivalent to
axiom (N4).

What is the essence of the classical octahedral axiom for triangulated
categories?  It starts with three given triangles
\begin{align*}
  & A_1 \to A_2 \to A_3 \to \Sigma A_1 \\
  & A_1 \to B_2 \to B_3 \to \Sigma A_1 \\
  & A_2 \to B_2 \to C_3 \to \Sigma A_2
\end{align*}
that are connected, in that each pair of triangles share a common
object. The axiom then guarantees the existence of two new morphisms,
and from these new morphisms we obtain three things:
\begin{enumerate}
\item A morphism of triangles.
\item A new triangle, whose objects are objects in the three original
  triangles.
\item Commutativity relations between morphisms.
\end{enumerate}
The reason why the axiom is called the ``octahedral axiom'' is that
everything fits into an octahedral whose vertices are the objects, and
where the edges are the morphisms.

The essence of the higher octahedral axiom for $n$-angulated
categories that we now introduce is exactly the same.  It starts with
three given $n$-angles, and guarantees the existence of $3n - 7$ new
morphisms.  From these new morphisms we obtain a morphism of
$n$-angles, a new $n$-angle and a certain commutativity relation
between morphisms.

\newpage

\begin{itemize}
\item[\textbf{(N4*)}] Given a commutative diagram
  \begin{center}
    \begin{tikzpicture}
      % Top row      
      \node (A1oppe) at (0,1.5){$A_1$};
      \node (A2) at (2,1.5){$A_2$};
      \node (A3) at (4,1.5){$A_3$};
      \node (Adots) at (6,1.5){$\cdots$};
      \node (An-1) at (8,1.5){$A_{n - 1}$};
      \node (An) at (10,1.5){$A_n$};
      \node (SA1oppe) at (12,1.5){$\Sigma A_1$};

      % Bottom row
      \node (A1nede) at (0,0){$A_1$};
      \node (B2) at (2,0){$B_2$};
      \node (B3) at (4,0){$B_3$};
      \node (Bdots) at (6,0){$\cdots$};
      \node (Bn-1) at (8,0){$B_{n - 1}$};
      \node (Bn) at (10,0){$B_n$};
      \node (SA1nede) at (12,0){$\Sigma A_1$};

      % Second column
      \node (C3) at (2,-1.5){$C_3$};
      \node (Cdots) at (2,-3){$\vdots$};
      \node (Cn-1) at (2,-4.5){$C_{n - 1}$};
      \node (Cn) at (2,-6){$C_n$};
      \node (SA2) at (2,-7.5){$\Sigma A_2$};

      \begin{scope}[font=\scriptsize,->,midway]
        % Vertical arrows
        \draw[-,double equal sign distance] (A1oppe) -- (A1nede);
        \draw (A2) -- node[right]{$\varphi_2$} (B2);
        \draw[-,double equal sign distance] (SA1oppe) -- (SA1nede);
        \draw (B2) -- node[right]{$\gamma_2$} (C3);
        \draw (C3) -- node[right]{$\gamma_3$} (Cdots); 
        \draw (Cdots) -- node[right]{$\gamma_{n - 2}$}( Cn-1);
        \draw (Cn-1) -- node[right]{$\gamma_{n - 1}$} (Cn);
        \draw (Cn) -- node[right]{$\gamma_n$} (SA2);

        % Horizontal arrows
        \draw (A1oppe) -- node[above]{$\alpha_1$} (A2);
        \draw (A2) -- node[above]{$\alpha_2$} (A3);
        \draw (A3) -- node[above]{$\alpha_3$} (Adots);
        \draw (Adots) -- node[above]{$\alpha_{n - 2}$} (An-1);
        \draw (An-1) -- node[above]{$\alpha_{n - 1}$} (An);
        \draw (An) -- node[above]{$\alpha_n$} (SA1oppe);
        \draw (A1nede) -- node[above]{$\beta_1$} (B2);
        \draw (B2) -- node[above]{$\beta_2$} (B3);
        \draw (B3) -- node[above]{$\beta_3$} (Bdots);
        \draw (Bdots) -- node[above]{$\beta_{n - 2}$} (Bn-1);
        \draw (Bn-1) -- node[above]{$\beta_{n - 1}$} (Bn);
        \draw (Bn) -- node[above]{$\beta_n$} (SA1nede);
      \end{scope}
    \end{tikzpicture}
  \end{center}
  whose top rows and second column are $n$-angles.  Then there exist
  morphisms $A_i \xrightarrow{\varphi_i} B_i$ ($3 \leq i \leq n$) and
  $\psi_j$ ($1 \leq j \leq 2n-5$) with the following two properties:
  \begin{enumerate}
  \item The sequence
    $(1,\varphi_2,\varphi_3,\ldots,\varphi_n)$ is a morphism
    of $n$-angles.
  \item The $n$-$\Sigma$-sequence
    \begin{center}
      \begin{tikzpicture}
        % Top row
        \node (A3) at (0,1.25){$A_3$};
        \node (A4B3) at (2,1.25){$A_4 \oplus B_3$};
        \node (A5B4C3) at (5,1.25){$A_5 \oplus B_4 \oplus C_3$};
        \node (A6B5C4) at (9,1.25){$A_6 \oplus B_5 \oplus C_4$};
        \node (tdots) at (12,1.25){$\cdots$};
        
        % Bottom row
        \node (mdots) at (0.25,0){$ $};
        \node (AnBn-1Cn-2) at (4.75,0){$A_n \oplus B_{n - 1} \oplus
          C_{n - 2}$};
        \node (BnCn-1) at (9.5,0){$B_n \oplus C_{n - 1}$};
        \node (mend) at (12.25,0){$C_n$};
        \node (mend2) at (14,0){$\Sigma A_3$};
        
        % Horizontal arrows
        \begin{scope}[font=\scriptsize,->,midway,above]
          % Top row
          \draw (A3) -- node{$\left[
              \begin{smallmatrix}
                \alpha_3\\
                \varphi_3
              \end{smallmatrix}
            \right]$} (A4B3);
          \draw (A4B3) -- node{$\left[
              \begin{smallmatrix}
                -\alpha_4 & 0\\
                \hfill \varphi_4 & -\beta_3\\
                \hfill \psi_2 & \hfill \psi_1
              \end{smallmatrix}
            \right]$} (A5B4C3);
          \draw (A5B4C3) -- node{$\left[
              \begin{smallmatrix}
                -\alpha_5 & 0 & 0\\
                -\varphi_5 & -\beta_4 & 0\\
                \hfill \psi_4 & \hfill \psi_3 & \gamma_3
              \end{smallmatrix}
            \right]$} (A6B5C4);
          \draw (A6B5C4) -- node{$\left[
              \begin{smallmatrix}
                -\alpha_6 & 0 & 0\\
                \hfill \varphi_6 & -\beta_5 & 0\\
                \hfill \psi_6 & \hfill \psi_5 & \gamma_4
              \end{smallmatrix}
            \right]$} (tdots);
          
          % Bottom row
          \draw (mdots) -- node{$\left[
              \begin{smallmatrix}
                -\alpha_{n - 1} & 0 & 0\\
                (-1)^{n + 1} \varphi_{n - 1} & -\beta_{n - 2} & 0\\
                \psi_{2n - 8} & \psi_{2n - 9} & \gamma_{n - 3}
              \end{smallmatrix}
            \right]$} (AnBn-1Cn-2);
          \draw (AnBn-1Cn-2) -- node{$\left[
              \begin{smallmatrix}
                (-1)^n \varphi_n & -\beta_{n - 1} & 0\\
                \psi_{2n - 6} & \psi_{2n - 7} & \gamma_{n - 2}
              \end{smallmatrix}
            \right]$} (BnCn-1);
          \draw (BnCn-1) -- node{$\left[
              \begin{smallmatrix}
                \psi_{2n - 5} & \gamma_{n - 1}
              \end{smallmatrix}
            \right]$} (mend);
          \draw (mend) -- node{$\Sigma \alpha_2 \circ \gamma_n$}
            (mend2);
        \end{scope}
      \end{tikzpicture}
    \end{center}
    is an $n$-angle, and $\gamma_n \circ \psi_{2n-5} = \Sigma \alpha_1
    \circ \beta_n$.
  \end{enumerate}
\end{itemize}

For small values of $n$, objects $A_i,B_i,C_i$ with $i > n$ appearing
in the axiom should be interpreted as zero objects (and so should
objects $C_i$ with $i < 3$).  Specifically, when $n = 3$, that is,
when $\C$ is a triangulated category, the triangle in (2) becomes
\begin{equation*}
  A_3 \xrightarrow{\varphi_3} B_3 \xrightarrow{\psi_1} C_3
  \xrightarrow{\Sigma \alpha_2 \circ \gamma_3} \Sigma A_3
\end{equation*}
and for $n = 4$, the $4$-angle in (2) becomes
\begin{equation*}
  A_3 \xrightarrow{\left[
      \begin{smallmatrix}
        \alpha_3\\
        \varphi_3
      \end{smallmatrix}
    \right]} A_4 \oplus B_3 \xrightarrow{\left[
      \begin{smallmatrix}
        \varphi_4 & -\beta_3\\
        \psi_2 & \hfill \psi_1
      \end{smallmatrix}
    \right]} B_4 \oplus C_3 \xrightarrow{\left[
      \begin{smallmatrix}
        \psi_3 & \gamma_3\\
      \end{smallmatrix}
    \right]} C_4 \xrightarrow{\Sigma \alpha_2 \circ \gamma_4} \Sigma
  A_3.
\end{equation*}

Our aim is to prove that axiom (N4) may be replaced by the new axiom
(N4*).  In other words, we shall prove that if our category $\C$ is
pre-triangulated (that is, $\C$ satisfies (N1), (N2) and (N3)), then
it satisfies (N4) if and only if it satisfies (N4*).  In order to
prove this, we need the following lemma.

\begin{lemma}
  \label{lem:hmtpy}
  Suppose $\C$ is $n$-angulated, and let
  \begin{center}
    \begin{tikzpicture}
      % Top row
      \node (A1) at (0,1.5){$A_1$};
      \node (A2) at (1.75,1.5){$A_2$};
      \node (A3) at (3.5,1.5){$A_3$};
      \node (Adots) at (5.25,1.5){$\cdots$};
      \node (An) at (7,1.5){$A_n$};
      \node (SA1) at (8.75,1.5){$\Sigma A_1$};
      
      % Bottom row
      \node (B1) at (0,0){$A_1$};
      \node (B2) at (1.75,0){$B_2$};
      \node (B3) at (3.5,0){$B_3$};
      \node (Bdots) at (5.25,0){$\cdots$};
      \node (Bn) at (7,0){$B_n$};
      \node (SB1) at (8.75,0){$\Sigma A_1$};
      
      \begin{scope}[font=\scriptsize,->,midway]
        % Vertical arrows
        \draw[-,double equal sign distance] (A1) -- (B1);
        \draw (A2) -- node[right]{$\varphi_2$} (B2);
        \draw[-,double equal sign distance] (SA1) -- (SB1);
        
        % Horizontal arrows
        \draw (A1) -- node[above]{$\alpha_1$} (A2);
        \draw (A2) -- node[above]{$\alpha_2$} (A3);
        \draw (A3) -- node[above]{$\alpha_3$} (Adots);
        \draw (Adots) -- node[above]{$\alpha_{n - 1}$} (An);
        \draw (An) -- node[above]{$\alpha_n$} (SA1);
        \draw (B1) -- node[above]{$\beta_1$} (B2);
        \draw (B2) -- node[above]{$\beta_2$} (B3);
        \draw (B3) -- node[above]{$\beta_3$} (Bdots);
        \draw (Bdots) -- node[above]{$\beta_{n - 1}$} (Bn);
        \draw (Bn) -- node[above]{$\beta_n$} (SB1);
      \end{scope}
    \end{tikzpicture}
  \end{center}
  be a commutative diagram whose rows are $n$-angles.  Apply axiom
  \emph{(N4)} and complete the diagram to a morphism
  \begin{center}
    \begin{tikzpicture}
      % Top row
      \node (A1) at (0,1.5){$A_1$};
      \node (A2) at (1.75,1.5){$A_2$};
      \node (A3) at (3.5,1.5){$A_3$};
      \node (Adots) at (5.25,1.5){$\cdots$};
      \node (An) at (7,1.5){$A_n$};
      \node (SA1) at (8.75,1.5){$\Sigma A_1$};

      % Bottom row
      \node (B1) at (0,0){$A_1$};
      \node (B2) at (1.75,0){$B_2$};
      \node (B3) at (3.5,0){$B_3$};
      \node (Bdots) at (5.25,0){$\cdots$};
      \node (Bn) at (7,0){$B_n$};
      \node (SB1) at (8.75,0){$\Sigma A_1$};

      \begin{scope}[font=\scriptsize,->,midway]
        % Vertical arrows
        \draw[-,double equal sign distance] (A1) -- (B1);
        \draw (A2) -- node[right]{$\varphi_2$} (B2);
        \draw (A3) -- node[right]{$\varphi_3$} (B3);
        \draw (An) -- node[right]{$\varphi_n$} (Bn);
        \draw[-,double equal sign distance] (SA1) -- (SB1);

        % Horizontal arrows
        \draw (A1) -- node[above]{$\alpha_1$} (A2);
        \draw (A2) -- node[above]{$\alpha_2$} (A3);
        \draw (A3) -- node[above]{$\alpha_3$} (Adots);
        \draw (Adots) -- node[above]{$\alpha_{n - 1}$} (An);
        \draw (An) -- node[above]{$\alpha_n$} (SA1);
        \draw (B1) -- node[above]{$\beta_1$} (B2);
        \draw (B2) -- node[above]{$\beta_2$} (B3);
        \draw (B3) -- node[above]{$\beta_3$} (Bdots);
        \draw (Bdots) -- node[above]{$\beta_{n - 1}$} (Bn);
        \draw (Bn) -- node[above]{$\beta_n$} (SB1);
      \end{scope}
    \end{tikzpicture}
  \end{center}
  of $n$-angles, in such a way that the mapping cone is also an
  $n$-angle.  Then the $n$-$\Sigma$-sequence
  \begin{align*}
    A_2 \xrightarrow{\left[
        \begin{smallmatrix}
          -\alpha_2\\
          \hfill \varphi_2
        \end{smallmatrix}
      \right]} A_3 \oplus B_2 &\xrightarrow{\left[
        \begin{smallmatrix}
          \alpha_3 & 0\\
          \varphi_3 & \beta_2
        \end{smallmatrix}
      \right]} A_4 \oplus B_3 \xrightarrow{\left[
        \begin{smallmatrix}
          \hfill \alpha_4 & 0\\
          -\varphi_4 & \beta_3
        \end{smallmatrix}
      \right]} \dots\\
    & \dots \xrightarrow{\left[
        \begin{smallmatrix}
          \alpha_{n - 1} & 0\\
          (-1)^n \varphi_{n - 1} & \beta_{n - 2}
        \end{smallmatrix}
      \right]} A_n \oplus B_{n - 1} \xrightarrow{\left[
        \begin{smallmatrix}
          (-1)^{n + 1} \varphi_n & \beta_{n - 1}
        \end{smallmatrix}
      \right]} B_n \xrightarrow{\Sigma\alpha_1 \circ \beta_n} \Sigma A_2
  \end{align*}  
  is an $n$-angle.
\end{lemma}

\begin{proof}
  The mapping cone is the middle $n$-$\Sigma$-sequence in the direct
  sum diagram
  \begin{center}
    \begin{tikzpicture}
      % Top row
      \node (A2) at (0,4){$A_2$};
      \node (A3+B2) at (2.6,4){$A_3 \oplus B_2$};
      \node (dots) at (4.4,4){$\cdots$};
      \node (An+Bn-1) at (7.5,4){$A_n \oplus B_{n - 1}$};
      \node (Bn) at (10.5,4){$B_n$};
      \node (SA2) at (13.25,4){$\Sigma A_2$};

      % Middle row
      \node (A2+A1) at (0,2){$A_2 \oplus A_1$};
      \node (A3+B2m) at (2.6,2){$A_3 \oplus B_2$};
      \node (dotsm) at (4.4,2){$\dots$};
      \node (An+Bn-1m) at (7.5,2){$A_n \oplus B_{n - 1}$};
      \node (SA1+Bn) at (10.5,2){$\Sigma A_1 \oplus B_n$};
      \node (SA2+SA1) at (13.25,2){$\Sigma A_2 \oplus \Sigma A_1$};

      % Bottom row
      \node (A2b) at (0,0){$A_2$};
      \node (A3+B2b) at (2.6,0){$A_3 \oplus B_2$};
      \node (dotsb) at (4.4,0){$\cdots$};
      \node (An+Bn-1b) at (7.5,0){$A_n \oplus B_{n - 1}$};
      \node (Bnb) at (10.5,0){$B_n$};
      \node (SA2b) at (13.25,0){$\Sigma A_2$};

      \begin{scope}[font=\scriptsize,->,midway]
        % Vertical arrows
        \begin{scope}[right]
          % Top
          \draw (A2) -- node{$\left[
              \begin{smallmatrix}
                1\\
                0
              \end{smallmatrix}
            \right]$} (A2+A1);
          \draw (A3+B2) -- node{$\left[
              \begin{smallmatrix}
                1 & 0\\
                0 & 1
              \end{smallmatrix}
            \right]$} (A3+B2m);
          \draw (An+Bn-1) -- node{$\left[
              \begin{smallmatrix}
                (-1)^{n + 1} & 0\\
                0 & 1
              \end{smallmatrix}
            \right]$} (An+Bn-1m);
          \draw (Bn) -- node{$\left[
              \begin{smallmatrix}
                - \beta_n\\
                1
              \end{smallmatrix}
            \right]$} (SA1+Bn);
          \draw (SA2) -- node{$\left[
              \begin{smallmatrix}
                1\\
                0
              \end{smallmatrix}
            \right]$} (SA2+SA1);
          
          % Bottom
          \draw (A2+A1) -- node{$\left[
              \begin{smallmatrix}
                1 & \alpha_1
              \end{smallmatrix}
            \right]$}(A2b);
          \draw (A3+B2m) -- node{$\left[
              \begin{smallmatrix}
                1 & 0\\
                0 & 1
              \end{smallmatrix}
            \right]$}(A3+B2b);
          \draw (An+Bn-1m) -- node{$\left[
              \begin{smallmatrix}
                (-1)^{n + 1} & 0\\
                0 & 1
              \end{smallmatrix}
            \right]$}(An+Bn-1b);
          \draw (SA1+Bn) -- node{$\left[
              \begin{smallmatrix}
                0 & 1
              \end{smallmatrix}
            \right]$}(Bnb);
          \draw (SA2+SA1) -- node{$\left[
              \begin{smallmatrix}
                1 & \Sigma \alpha_1 
              \end{smallmatrix}
            \right]$}(SA2b);
        \end{scope}

        % Horizontal arrows
        \begin{scope}[above]
          % Top row
          \draw (A2) -- node{$\left[
              \begin{smallmatrix}
                -\alpha_2\\
                \hfill \varphi_2
              \end{smallmatrix}
            \right]$} (A3+B2);
          \draw (A3+B2) -- node{$\left[
              \begin{smallmatrix}
                \alpha_3 & 0\\
                \varphi_3 & \beta_2
              \end{smallmatrix}
            \right]$} (dots);
          \draw (dots) -- node{$\left[
              \begin{smallmatrix}
                \alpha_{n - 1} & 0\\
                (-1)^n \varphi_{n - 1} & \beta_{n - 2}
              \end{smallmatrix}
            \right]$} (An+Bn-1);
          \draw (An+Bn-1) -- node{$\left[
              \begin{smallmatrix}
                (-1)^{n + 1} \varphi_n & \beta_{n - 1}
              \end{smallmatrix}
            \right]$} (Bn);
          \draw (Bn) -- node{$\Sigma \alpha_1 \circ \beta_n$} (SA2);

          % Middle row
          \draw (A2+A1) -- node{$\left[
              \begin{smallmatrix}
                -\alpha_2 & 0\\
                \hfill \varphi_2 & \beta_1
              \end{smallmatrix}
            \right]$} (A3+B2m);
          \draw (A3+B2m) -- node{$\left[
              \begin{smallmatrix}
                -\alpha_3 & 0\\
                \hfill \varphi_3 & \beta_2
              \end{smallmatrix}
            \right]$} (dotsm);
          \draw (dotsm) -- node{$\left[
              \begin{smallmatrix}
                -\alpha_{n - 1} & 0\\
                \hfill \varphi_{n - 1} & \beta_{n - 2}
              \end{smallmatrix}
            \right]$} (An+Bn-1m);
          \draw (An+Bn-1m) -- node{$\left[
              \begin{smallmatrix}
                -\alpha_n & 0\\
                \hfill \varphi_n & \beta_{n - 1}
              \end{smallmatrix}
            \right]$} (SA1+Bn);
          \draw (SA1+Bn) -- node{$\left[
              \begin{smallmatrix}
                \hfill -\Sigma\alpha_1 & 0\\
                1 & \beta_n
              \end{smallmatrix}
            \right]$} (SA2+SA1);

          % Bottom row
          \draw (A2b) -- node{$\left[
              \begin{smallmatrix}
                -\alpha_2\\
                \hfill \varphi_2
              \end{smallmatrix}
            \right]$} (A3+B2b);
          \draw (A3+B2b) -- node{$\left[
              \begin{smallmatrix}
                \alpha_3 & 0\\
                \varphi_3 & \beta_2
              \end{smallmatrix}
            \right]$} (dotsb);
          \draw (dotsb) -- node{$\left[
              \begin{smallmatrix}
                \alpha_{n - 1} & 0\\
                (-1)^n \varphi_{n - 1} & \beta_{n - 2}
              \end{smallmatrix}
            \right]$} (An+Bn-1b);
          \draw (An+Bn-1b) -- node{$\left[
              \begin{smallmatrix}
                (-1)^{n + 1} \varphi_n & \beta_{n - 1}
              \end{smallmatrix}
            \right]$} (Bnb);
          \draw (Bnb) -- node{$\Sigma \alpha_1 \circ \beta_n$} (SA2b);
        \end{scope}
      \end{scope}
    \end{tikzpicture}
  \end{center}
  Therefore, by axiom (N1)(a), the top (bottom) row is also an
  $n$-angle.
\end{proof}

Now we prove that axioms (N4) and (N4*) are equivalent.  We do this in
two steps, showing first that axiom (N4) implies axiom (N4*).

\begin{theorem}
  \label{thm:octa}
  If $\nang$ is a collection of $n$-$\Sigma$-sequences in $\C$
  satisfying axioms \emph{(N1), (N2), (N3)} and \emph{(N4)}, then it
  also satisfies \emph{(N4*)}.
\end{theorem}

\begin{proof}
  Suppose we are given a commutative diagram
  \begin{center}
    \begin{tikzpicture}
      % Top row      
      \node (A1oppe) at (0,1.5){$A_1$};
      \node (A2) at (2,1.5){$A_2$};
      \node (A3) at (4,1.5){$A_3$};
      \node (Adots) at (6,1.5){$\cdots$};
      \node (An-1) at (8,1.5){$A_{n - 1}$};
      \node (An) at (10,1.5){$A_n$};
      \node (SA1oppe) at (12,1.5){$\Sigma A_1$};

      % Bottom row
      \node (A1nede) at (0,0){$A_1$};
      \node (B2) at (2,0){$B_2$};
      \node (B3) at (4,0){$B_3$};
      \node (Bdots) at (6,0){$\cdots$};
      \node (Bn-1) at (8,0){$B_{n - 1}$};
      \node (Bn) at (10,0){$B_n$};
      \node (SA1nede) at (12,0){$\Sigma A_1$};

      \begin{scope}[font=\scriptsize,->,midway]
        % Vertical arrows
        \draw[-,double equal sign distance] (A1oppe) -- (A1nede);
        \draw (A2) -- node[right]{$\varphi_2$} (B2);
        \draw[-,double equal sign distance] (SA1oppe) -- (SA1nede);
        
        % Horizontal arrows
        \draw (A1oppe) -- node[above]{$\alpha_1$} (A2);
        \draw (A2) -- node[above]{$\alpha_2$} (A3);
        \draw (A3) -- node[above]{$\alpha_3$} (Adots);
        \draw (Adots) -- node[above]{$\alpha_{n - 2}$} (An-1);
        \draw (An-1) -- node[above]{$\alpha_{n - 1}$} (An);
        \draw (An) -- node[above]{$\alpha_n$} (SA1oppe);
        \draw (A1nede) -- node[above]{$\beta_1$} (B2);
        \draw (B2) -- node[above]{$\beta_2$} (B3);
        \draw (B3) -- node[above]{$\beta_3$} (Bdots);
        \draw (Bdots) -- node[above]{$\beta_{n - 2}$} (Bn-1);
        \draw (Bn-1) -- node[above]{$\beta_{n - 1}$} (Bn);
        \draw (Bn) -- node[above]{$\beta_n$} (SA1nede);
      \end{scope}
    \end{tikzpicture}
  \end{center}
  where the two rows are $n$-angles, and where the map $\varphi_1$ is
  an isomorphism. Furthermore, let
  \begin{equation*}
    A_2 \xrightarrow{\varphi_2} B_2 \xrightarrow{\gamma_2} C_3 
    \xrightarrow{\gamma_3} \cdots \xrightarrow{\gamma_{n-1}} C_n
    \xrightarrow{\gamma_n} \Sigma A_2
  \end{equation*}
  be an $n$-angle.  Apply axiom (N4) and complete the given diagram to
  a morphism $(1,\varphi_2,\varphi_3,\ldots,\varphi_n)$ of
  $n$-angles, in such a way that the mapping cone is an $n$-angle.
  Then the first part of axiom (N4*) is already satisfied.

  By Lemma \ref{lem:hmtpy}, the $n$-$\Sigma$-sequence
  \begin{equation*}
    A_2 \xrightarrow{\left[
        \begin{smallmatrix}
          -\alpha_2\\
          \hfill \varphi_2
        \end{smallmatrix}
      \right]} A_3 \oplus B_2 \xrightarrow{\left[
        \begin{smallmatrix}
          \alpha_3 & 0\\
          \varphi_3 & \beta_2
        \end{smallmatrix}
      \right]} A_4 \oplus B_3 \xrightarrow{\left[
        \begin{smallmatrix}
          \hfill \alpha_4 & 0\\
          -\varphi_4 & \beta_3
        \end{smallmatrix}
      \right]} \dots
    \xrightarrow{\left[
        \begin{smallmatrix}
          \alpha_{n - 1} & 0\\
          (-1)^n \varphi_{n - 1} & \beta_{n - 2}
        \end{smallmatrix}
      \right]} A_n \oplus B_{n - 1} \xrightarrow{\left[
        \begin{smallmatrix}
          (-1)^{n + 1} \varphi_n & \beta_{n - 1}
        \end{smallmatrix}
      \right]} B_n \xrightarrow{\Sigma \alpha_1 \circ \beta_n} \Sigma A_2
  \end{equation*}
  is an $n$-angle.  Then by axiom (N4) again, there exist morphisms
  $\psi_3,\psi_4,\ldots,\psi_{2n - 5}$ such that the mapping cone of
  the morphism
  \begin{center}
    \begin{tikzpicture}
      % Top row
      \node (A2oppe) at (0,1.5){$A_2$};
      \node (A3B2) at (2,1.5){$A_3 \oplus B_2$};
      \node (A4B3) at (4.5,1.5){$A_4 \oplus B_3$};
      \node (Adots) at (6.5,1.5){$\cdots$};
      \node (AnBn-1) at (9.5,1.5){$A_n \oplus B_{n - 1}$};
      \node (Bn) at (12.5,1.5){$B_n$};
      \node (SA2oppe) at (15,1.5){$\Sigma A_2$};

      % Bottom row
      \node (A2nede) at (0,0){$A_2$};
      \node (B2) at (2,0){$B_2$};
      \node (C3) at (4.5,0){$C_3$};
      \node (Cdots) at (6.5,0){$\cdots$};
      \node (Cn-1) at (9.5,0){$C_{n - 1}$};
      \node (Cn) at (12.5,0){$C_n$};
      \node (SA2nede) at (15,0){$\Sigma A_2$};

      % Vertical arrows
      \begin{scope}[font=\scriptsize,->,midway,right]
        \draw[-,double equal sign distance] (A2oppe) -- (A2nede);
        \draw (A3B2) -- node{$\left[
            \begin{smallmatrix}
              0 & 1
            \end{smallmatrix}
          \right]$} (B2);
        \draw (A4B3) -- node{$\left[
            \begin{smallmatrix}
              \psi_2 & \psi_1
            \end{smallmatrix}
          \right]$} (C3);
        \draw (AnBn-1) -- node{$\left[
            \begin{smallmatrix}
              \psi_{2n - 6} & \psi_{2n - 7}
            \end{smallmatrix}
          \right]$} (Cn-1);
        \draw (Bn) -- node{$\psi_{2n - 5}$} (Cn);
        \draw[-,double equal sign distance] (SA2oppe) -- (SA2nede);
      \end{scope}

      % Horizontal arrows
      \begin{scope}[font=\scriptsize,->,midway,above]
        \draw (A2oppe) -- node{$\left[
            \begin{smallmatrix}
              -\alpha_2\\
              \hfill \varphi_2
            \end{smallmatrix}
          \right]$} (A3B2);
        \draw (A3B2) -- node{$\left[
            \begin{smallmatrix}
              \alpha_3 & 0\\
              \varphi_3 & \beta_2
            \end{smallmatrix}
          \right]$} (A4B3);
        \draw (A4B3) -- node{$\left[
            \begin{smallmatrix}
              \hfill \alpha_4 & 0\\
              -\varphi_4 & \beta_3
            \end{smallmatrix}
          \right]$} (Adots);
        \draw (Adots) -- node{$\left[
            \begin{smallmatrix}
              \alpha_{n - 1} & 0\\
              (-1)^n \varphi_{n - 1} & \beta_{n - 2}
            \end{smallmatrix}
          \right]$} (AnBn-1);
        \draw (AnBn-1) -- node{$\left[
            \begin{smallmatrix}
              (-1)^{n + 1} \varphi_n & \beta_{n - 1}
            \end{smallmatrix}
          \right]$} (Bn);
        \draw (Bn) -- node{$\Sigma \alpha_1 \circ \beta_n$} (SA2oppe);
        \draw (A2nede) -- node{$\varphi_2$} (B2);
        \draw (B2) -- node{$\gamma_2$} (C3);
        \draw (C3) -- node{$\gamma_3$} (Cdots);
        \draw (Cdots) -- node{$\gamma_{n - 2}$} (Cn-1);
        \draw (Cn-1) -- node{$\gamma_{n - 1}$} (Cn);
        \draw (Cn) -- node{$\gamma_n$} (SA2nede);
      \end{scope}
    \end{tikzpicture}
  \end{center}
  is an $n$-angle.  In other words, the $n$-$\Sigma$-sequence
  \begin{align*}
    & A_3 \oplus B_2 \oplus A_2 \xrightarrow{\left[
        \begin{smallmatrix}
          -\alpha_3 & 0 & 0\\
          -\varphi_3 & -\beta_2 & 0\\
          0 & 1 & \varphi_2
        \end{smallmatrix}
      \right]} A_4 \oplus B_3 \oplus B_2 \xrightarrow{\mu_1} 
    A_5 \oplus B_4 \oplus C_3 \xrightarrow{\mu_2} \cdots\\
    &  \xrightarrow{\mu_{n-4}} A_n \oplus B_{n - 1} \oplus C_{n - 2}
    \xrightarrow{\left[
        \begin{smallmatrix}
          (-1)^n \varphi_n & -\beta_{n - 1} & 0\\
          \psi_{2n - 6} & \psi_{2n - 7} & \gamma_{n - 2}
        \end{smallmatrix}
      \right]} B_n \oplus C_{n - 1} \xrightarrow{\left[
        \begin{smallmatrix}
          -\Sigma \alpha_1 \circ \beta_n & 0\\
          \psi_{2n - 5} & \gamma_{n - 1}
        \end{smallmatrix}
      \right]}
    \Sigma A_2 \oplus C_n \xrightarrow{\left[
        \begin{smallmatrix}
          \hfill \Sigma \alpha_2 & 0\\
          -\Sigma \varphi_2 & 0\\
          1 & \gamma_n
        \end{smallmatrix}
      \right]} \Sigma A_3 \oplus \Sigma B_2 \oplus \Sigma A_2
  \end{align*}
  is an $n$-angle, where $\mu_i$ is the matrix
  \begin{equation*}
    \mu_i = \left[
      \begin{smallmatrix}
        - \alpha_{i+3} & 0 & 0 \\
        (-1)^{i+1} \varphi_{i+3} & - \beta_{i+2} & 0 \\
        \psi_{2i} & \psi_{2i - 1} & \gamma_{i+1}
      \end{smallmatrix}
    \right].
  \end{equation*}
  This $n$-angle is the middle $n$-$\Sigma$-sequence in the
  direct sum diagram
  \begin{center}
    \begin{tikzpicture}
      % First part
      \begin{scope}
        % Top row
        \node (A3oppe) at (0,4.5){$A_3$};
        \node (A4B3oppe) at (3.5,4.5){$A_4 \oplus B_3$};
        \node (A5B4C3oppe) at (7,4.5){$A_5 \oplus B_4 \oplus C_3$};
        \node (A6B5C4oppe) at (10.5,4.5){$A_6 \oplus B_5 \oplus C_4$};
        \node (dotsoppe) at (13.5,4.5){$\cdots$};
        
        % Middle row
        \node (A3B2A2) at (0,2.25){$A_3 \oplus B_2 \oplus A_2$};
        \node (A4B3B2) at (3.5,2.25){$A_4 \oplus B_3 \oplus B_2$};
        \node (A5B4C3) at (7,2.25){$A_5 \oplus B_4 \oplus C_3$};
        \node (A6B5C4) at (10.5,2.25){$A_6 \oplus B_5 \oplus C_4$};
        \node (dots) at (13.5,2.25){$\cdots$};
        
        % Bottom row
        \node (A3nede) at (0,0){$A_3$};
        \node (A4B3nede) at (3.5,0){$A_4 \oplus B_3$};
        \node (A5B4C3nede) at (7,0){$A_5 \oplus B_4 \oplus C_3$};
        \node (A6B5C4nede) at (10.5,0){$A_6 \oplus B_5 \oplus C_4$};
        \node (dotsnede) at (13.5,0){$\cdots$};
        
        % Vertical arrows
        \begin{scope}[font=\scriptsize,->,midway,right]
          \draw (A3oppe) -- node{$\left[
              \begin{smallmatrix}
                -1\\
                \hfill 0\\
                \hfill 0
              \end{smallmatrix}
            \right]$} (A3B2A2);
          \draw (A4B3oppe) -- node{$\left[
              \begin{smallmatrix}
                1 & 0\\
                0 & 1\\
                0 & 0
              \end{smallmatrix}
            \right]$} (A4B3B2);
          \draw[-,double equal sign distance] (A5B4C3oppe) --
            (A5B4C3);
          \draw[-,double equal sign distance] (A6B5C4oppe) --
            (A6B5C4);
          \draw (A3B2A2) -- node{$\left[
              \begin{smallmatrix}
                -1 & 0 & \alpha_2
              \end{smallmatrix}
            \right]$} (A3nede);
          \draw (A4B3B2) -- node{$\left[
              \begin{smallmatrix}
                1 & 0 & 0\\
                0 & 1 & \beta_2
              \end{smallmatrix}
            \right]$} (A4B3nede);
          \draw[-,double equal sign distance] (A5B4C3) --
            (A5B4C3nede);
          \draw[-,double equal sign distance] (A6B5C4) --
            (A6B5C4nede);
        \end{scope}
        
        % Horizontal arrows
        \begin{scope}[font=\scriptsize,->,midway,above]
          \draw (A3oppe) -- node{$\left[
              \begin{smallmatrix}
                \alpha_3\\
                \varphi_3
              \end{smallmatrix}
            \right]$} (A4B3oppe);
          \draw (A4B3oppe) -- node{$\left[
              \begin{smallmatrix}
                -\alpha_4 & 0\\
                \hfill \varphi_4 & -\beta_3\\
                \hfill \psi_2 & \hfill \psi_1
              \end{smallmatrix}
            \right]$} (A5B4C3oppe);
          \draw (A5B4C3oppe) -- node{$\mu_2$} (A6B5C4oppe);
          \draw (A6B5C4oppe) -- node{$\mu_3$} (dotsoppe);
          
          \draw (A3B2A2) -- node{$\left[
              \begin{smallmatrix}
                -\alpha_3 & 0 & 0\\
                -\varphi_3 & -\beta_2 & 0\\
                0 & 1 & \varphi_2
              \end{smallmatrix}
            \right]$} (A4B3B2);
          \draw (A4B3B2) -- node{$\mu_1$} (A5B4C3);
          \draw (A5B4C3) -- node{$\mu_2$} (A6B5C4);
          \draw (A6B5C4) -- node{$\mu_3$} (dots);
          
          \draw (A3nede) -- node{$\left[
              \begin{smallmatrix}
                \alpha_3\\
                \varphi_3
              \end{smallmatrix}
            \right]$} (A4B3nede);
          \draw (A4B3nede) -- node{$\left[
              \begin{smallmatrix}
                -\alpha_4 & 0\\
                \hfill \varphi_4 & -\beta_3\\
                \hfill \psi_2 & \hfill \psi_1
              \end{smallmatrix}
            \right]$} (A5B4C3nede);
          \draw (A5B4C3nede) -- node{$\mu_2$} (A6B5C4nede);
          \draw (A6B5C4nede) -- node{$\mu_3$} (dotsnede);
        \end{scope}
      \end{scope}

      % Second part
      \begin{scope}[xshift=0.25cm,yshift=-6cm]
        % Top row
        \node (dots2oppe) at (-0.5,4){};
        \node (AnBn-1Cn-2oppe) at (2,4){$A_n \oplus B_{n - 1} \oplus
          C_{n - 2}$};
        \node (BnCn-1oppe) at (6.5,4){$B_n \oplus C_{n - 1}$};

        \node (Cnoppe) at (9.5,4){$C_n$};
        \node (SA3oppe) at (13,4){$\Sigma A_3$};
        
        % Square
        \node (Omega) at (8,3){$\Omega$};

        % Middle row
        \node (dots2) at (-0.5,2){};
        \node (AnBn-1Cn-2) at (2,2){$A_n \oplus B_{n - 1} \oplus
          C_{n - 2}$};
        \node (BnCn-1) at (6.5,2){$B_n \oplus C_{n - 1}$};
        
        \node (SA2Cn) at (9.5,2){$\Sigma A_2 \oplus C_n$};
        \node (SA3SB2SA2) at (13,2){$\Sigma A_3 \oplus \Sigma B_2
          \oplus \Sigma A_2$};

        % Bottom row
        \node (dots2nede) at (-0.5,0){};
        \node (AnBn-1Cn-2nede) at (2,0){$A_n \oplus B_{n - 1} \oplus
          C_{n - 2}$};
        \node (BnCn-1nede) at (6.5,0){$B_n \oplus C_{n - 1}$};

        \node (Cnnede) at (9.5,0){$C_n$};
        \node (SA3nede) at (13,0){$\Sigma A_3$};

        % Vertical arrows
        \begin{scope}[font=\scriptsize,->,midway,right]
          \draw[-,double equal sign distance] (AnBn-1Cn-2oppe) --
            (AnBn-1Cn-2);
          \draw[-,double equal sign distance] (BnCn-1oppe) --
            (BnCn-1);

          \draw[-,double equal sign distance] (AnBn-1Cn-2) --
            (AnBn-1Cn-2nede);
          \draw[-,double equal sign distance] (BnCn-1) --
            (BnCn-1nede);
        \end{scope}

        % Horizontal arrows
        \begin{scope}[font=\scriptsize,->,midway,above]
          \draw (dots2oppe) -- node{$\mu_{n-4}$} (AnBn-1Cn-2oppe);
          \draw (AnBn-1Cn-2oppe) -- node{$\left[
              \begin{smallmatrix}
                (-1)^n \varphi_n & -\beta_{n - 1} & 0\\
                \psi_{2n - 6} & \psi_{2n - 7} & \gamma_{n - 2}
              \end{smallmatrix}
            \right]$}(BnCn-1oppe);
          \draw (BnCn-1oppe) -- node{$\left[
              \begin{smallmatrix}
                \psi_{2n - 5} & \gamma_{n - 1}
              \end{smallmatrix}
            \right]$} (Cnoppe);
          \draw (Cnoppe) -- node{$\Sigma \alpha_2 \circ \gamma_n$}
            (SA3oppe);

          \draw (dots2) -- node{$\mu_{n-4}$} (AnBn-1Cn-2);
          \draw (AnBn-1Cn-2) -- node{$\left[
              \begin{smallmatrix}
                (-1)^n \varphi_n & -\beta_{n - 1} & 0\\
                \psi_{2n - 6} & \psi_{2n - 7} & \gamma_{n - 2}
              \end{smallmatrix}
            \right]$}(BnCn-1);
          \draw (BnCn-1) -- node{$\left[
              \begin{smallmatrix}
                -\Sigma \alpha_1 \circ \beta_n & 0\\
                \psi_{2n - 5} & \gamma_{n - 1}
              \end{smallmatrix}
            \right]$}(SA2Cn);
          \draw (SA2Cn) -- node{$\left[
              \begin{smallmatrix}
                \hfill \Sigma \alpha_2 & 0\\
                -\Sigma \varphi_2 & 0\\
                1 & \gamma_n
              \end{smallmatrix}
            \right]$} (SA3SB2SA2);

          \draw (dots2nede) -- node{$\mu_{n - 4}$} (AnBn-1Cn-2nede);
          \draw (AnBn-1Cn-2nede) -- node{$\left[
              \begin{smallmatrix}
                (-1)^n \varphi_n & -\beta_{n - 1} & 0\\
                \psi_{2n - 6} & \psi_{2n - 7} & \gamma_{n - 2}
              \end{smallmatrix}
            \right]$}(BnCn-1nede);
          \draw (BnCn-1nede) -- node{$\left[
              \begin{smallmatrix}
                \psi_{2n - 5} & \gamma_{n - 1}
              \end{smallmatrix}
            \right]$} (Cnnede);
          \draw (Cnnede) -- node{$\Sigma \alpha_2 \circ \gamma_n$}
            (SA3nede);
        \end{scope}
      \end{scope}

      % Tail end
      \begin{scope}[xshift=0.25cm,yshift=-6cm]
        % Vertical arrows
        \begin{scope}[font=\scriptsize,->,midway,right]
          \draw (Cnoppe) -- node{$\left[
              \begin{smallmatrix}
                -\gamma_n\\
                1
              \end{smallmatrix}
            \right]$} (SA2Cn);
          \draw (SA3oppe) -- node{$\left[
              \begin{smallmatrix}
                -1\\
                \hfill 0\\
                \hfill 0
              \end{smallmatrix}
            \right]$} (SA3SB2SA2);

          \draw (SA2Cn) -- node{$\left[
              \begin{smallmatrix}
                0 & 1
              \end{smallmatrix}
            \right]$} (Cnnede);
          \draw (SA3SB2SA2) -- node{$\left[
              \begin{smallmatrix}
                -1 & 0 & \Sigma \alpha_2
              \end{smallmatrix}
            \right]$} (SA3nede);
        \end{scope}
      \end{scope}
     \end{tikzpicture}
  \end{center}
  Consequently, by (N1)(a), the top (bottom) $n$-$\Sigma$-sequence is
  an $n$-angle.  Moreover, the commutativity of the square $\Omega$
  implies that $\gamma_n \circ \psi_{2n-5} = \Sigma \alpha_1 \circ
  \beta_n$.  This shows that the second part of axiom (N4*) is
  satisfied.
\end{proof}

We now prove the converse to Theorem \ref{thm:octa}, namely that the
octahedral axiom (N4*) implies axiom (N4).

\begin{theorem}
  \label{thm:opposite}
  If $\nang$ is a collection of $n$-$\Sigma$-sequences in $\C$
  satisfying axioms \emph{(N1), (N2), (N3)} and \emph{(N4*)}, then it
  also satisfies \emph{(N4)}.
\end{theorem}

\begin{proof}
  Given a commutative diagram
  \begin{center}
    \begin{tikzpicture}
      % Top row
      \node (A1) at (0,1.5){$A_1$};
      \node (A2) at (1.5,1.5){$A_2$};
      \node (A3) at (3,1.5){$A_3$};
      \node (Adots) at (4.5,1.5){$\cdots$};
      \node (An) at (6,1.5){$A_n$};
      \node (SA1) at (7.5,1.5){$\Sigma A_1$};

      % Bottom row
      \node (B1) at (0,0){$B_1$};
      \node (B2) at (1.5,0){$B_2$};
      \node (B3) at (3,0){$B_3$};
      \node (Bdots) at (4.5,0){$\cdots$};
      \node (Bn) at (6,0){$B_n$};
      \node (SB1) at (7.5,0){$\Sigma B_1$};

      \begin{scope}[font=\scriptsize,->,midway]
        % Horizontal arrows
        \begin{scope}[above]
          \draw (A1) -- node{$\alpha_1$} (A2);
          \draw (A2) -- node{$\alpha_2$} (A3);
          \draw (A3) -- node{$\alpha_3$} (Adots);
          \draw (Adots) -- node{$\alpha_{n - 1}$} (An);
          \draw (An) -- node{$\alpha_n$} (SA1);
          \draw (B1) -- node{$\beta_1$} (B2);
          \draw (B2) -- node{$\beta_2$} (B3);
          \draw (B3) -- node{$\beta_3$} (Bdots);
          \draw (Bdots) -- node{$\beta_{n - 1}$} (Bn);
          \draw (Bn) -- node{$\beta_n$} (SB1);
        \end{scope}

        % Vertical arrows
        \begin{scope}[right]
          \draw (A1) -- node{$\varphi_1$} (B1);
          \draw (A2) -- node{$\varphi_2$} (B2);
          \draw (SA1) -- node{$\Sigma \varphi_1$} (SB1);
        \end{scope}
      \end{scope}
    \end{tikzpicture}
  \end{center}
  where the two rows are $n$-angles: we denote these by $A_\bullet$
  and $B_\bullet$.  We want to prove that we can complete the above
  diagram to a morphism of $n$-angles in such a way that the mapping
  cone of that morphism is again an $n$-angle.

  From the given diagram we build the diagram
  \begin{center}
    \begin{tikzpicture}
      % Top row
      \node (A1B1oppe) at (0,9){$A_1\oplus B_1$};
      \node (B2A2B1) at (3,9){$B_2\oplus A_2\oplus B_1$};
      \node (A3B2) at (5.75,9){$A_3\oplus B_2$};
      \node (A4) at (7.75,9){$A_4$};
      \node (Adotsoppe) at (9.25,9){$\cdots$};
      
      % Top row (continued)
      \node (Adotsnede) at (4,6){$\cdots$};
      \node (An-1) at (5.5,6){$A_{n - 1}$};
      \node (An) at (7.75,6){$A_n$};
      \node (SA1SB1oppe) at (12,6){$\Sigma A_1\oplus \Sigma B_1$};
      
      % Bottom row
      \node (A1B1nede) at (0,7.5){$A_1\oplus B_1$};
      \node (B2) at (3,7.5){$B_2$};
      \node (B3) at (5.75,7.5){$B_3$};
      \node (B4) at (7.75,7.5){$B_4$};
      \node (Bdotsoppe) at (9.25,7.5){$\cdots$};
      
      % Bottom row (continued)
      \node (Bdotsnede) at (4,4.5){$\cdots$};
      \node (Bn-1) at (5.5,4.5){$B_{n - 1}$};
      \node (SA1Bn) at (7.75,4.5){$\Sigma A_1\oplus B_n$};
      \node (SA1SB1nede) at (12,4.5){$\Sigma A_1\oplus \Sigma B_1$};
      
      % Second column
      \node (03) at (3,6){$0$};
      \node (dots) at (3,4.5){$\vdots$};
      \node (0n-2) at (3,3){$0$};
      \node (SA2SB1) at (3,1.5){$\Sigma A_2\oplus \Sigma B_1$};
      \node (SB2SA2SB1) at (3,0){$\Sigma B_2\oplus \Sigma
        A_2\oplus \Sigma B_1$};
      
      \begin{scope}[font=\scriptsize,->,midway]
        % Horizontal arrows
        \begin{scope}[above]
          \draw (A1B1oppe) -- node{$\left[
              \begin{smallmatrix}
                0 & \hfill 0\\
                (-1)^n \alpha_1 & \hfill 0\\
                0 & -1
              \end{smallmatrix}
            \right]$} (B2A2B1);
          \draw (B2A2B1) -- node{$\left[
              \begin{smallmatrix}
                0 & -\alpha_2 & 0\\
                1 & 0 & 0
              \end{smallmatrix}
            \right]$} (A3B2);
          \draw (A3B2) -- node{$\left[
              \begin{smallmatrix}
                -\alpha_3 & 0
              \end{smallmatrix}
            \right]$} (A4);
          \draw (A4) -- node{$\alpha_4$} (Adotsoppe);
          \draw (Adotsnede) -- node{$\alpha_{n - 2}$} (An-1);
          \draw (An-1) -- node{$\alpha_{n - 1}$} (An);
          \draw (An) -- node{$\left[
              \begin{smallmatrix}
                (-1)^n \alpha_n\\
                0
              \end{smallmatrix}
            \right]$} (SA1SB1oppe);
          \draw (A1B1nede) -- node{$\left[
              \begin{smallmatrix}
                (-1)^{n + 1} \varphi_2 \circ \alpha_1 & \beta_1
              \end{smallmatrix}
            \right]$} (B2);
          \draw (B2) -- node{$\beta_2$} (B3);
          \draw (B3) -- node{$-\beta_3$} (B4);
          \draw (B4) -- node{$-\beta_4$} (Bdotsoppe);
          \draw (Bdotsnede) -- node{$-\beta_{n - 2}$} (Bn-1);
          \draw (Bn-1) -- node{$\left[
              \begin{smallmatrix}
                0\\
                -\beta_{n - 1}
              \end{smallmatrix}
            \right]$} (SA1Bn);
          \draw (SA1Bn) -- node{$\left[
              \begin{smallmatrix}
                -1 & 0\\
                (-1)^{n + 1} \Sigma \varphi_1 & (-1)^{n + 1} \beta_n
              \end{smallmatrix}
            \right]$} (SA1SB1nede);
        \end{scope}
        
        % Vertical arrows
        \begin{scope}[right]
          \draw[-,double equal sign distance] (A1B1oppe) --
            (A1B1nede);
          \draw (B2A2B1) -- node{$\left[
              \begin{smallmatrix}
                1 & -\varphi_2 & -\beta_1
              \end{smallmatrix}
            \right]$} (B2);
          \draw (B2) -- (03);
          \draw (03) -- (dots);
          \draw (dots) -- (0n-2);
          \draw (0n-2) -- (SA2SB1);
          \draw (SA2SB1) -- node{$\left[
              \begin{smallmatrix}
                (-1)^n \Sigma \varphi_2 & (-1)^n \Sigma \beta_1\\
                (-1)^n & 0\\
                0 & (-1)^n
              \end{smallmatrix}
            \right]$} (SB2SA2SB1);
          \draw[-,double equal sign distance] (SA1SB1oppe) --
            (SA1SB1nede);
        \end{scope}
      \end{scope}
      \end{tikzpicture}
  \end{center}
  in which the top left square commutes.  Let $X_\bullet$, $Y_\bullet$
  and $Z_\bullet$ denote the three $n$-$\Sigma$-sequences
  \begin{align*}
    & B_2\oplus A_2\oplus B_1 \xrightarrow{\left[
        \begin{smallmatrix}
          1 & -\varphi_2 & -\beta_1
        \end{smallmatrix}
      \right]} B_2 \to 0 \to \cdots \to 0 \to \Sigma A_2\oplus \Sigma
    B_1 \xrightarrow{\left[
        \begin{smallmatrix}
          (-1)^n \Sigma \varphi_2 & (-1)^n \Sigma \beta_1\\
          (-1)^n & 0\\
          0 & (-1)^n
        \end{smallmatrix}
      \right]} \Sigma B_2\oplus \Sigma A_2\oplus \Sigma B_1,\\
    & A_1\oplus B_1 \xrightarrow{\left[
        \begin{smallmatrix}
          0 & \hfill 0\\
          (-1)^n \alpha_1 & \hfill 0\\
          0 & -1
        \end{smallmatrix}
      \right]} B_2\oplus A_2\oplus B_1 \xrightarrow{\left[
        \begin{smallmatrix}
          0 & -\alpha_2 & 0\\
          1 & 0 & 0
        \end{smallmatrix}
      \right]} A_3\oplus B_2 \xrightarrow{\left[
        \begin{smallmatrix}
          -\alpha_3 & 0
        \end{smallmatrix}
      \right]} A_4 \xrightarrow{\alpha_4} \cdots
    \xrightarrow{\alpha_{n - 1}} A_n \xrightarrow{\left[
        \begin{smallmatrix}
          (-1)^n \alpha_n\\
          0
        \end{smallmatrix}
      \right]} \Sigma A_1\oplus \Sigma B_1,\\
    & A_1\oplus B_1 \xrightarrow{\left[
        \begin{smallmatrix}
          (-1)^{n + 1} \varphi_2 \circ \alpha_1 & \beta_1
        \end{smallmatrix}
      \right]} B_2 \xrightarrow{\beta_2} B_3 \xrightarrow{-\beta_3}
    \cdots \xrightarrow{-\beta_{n - 2}} B_{n - 1} \xrightarrow{\left[
        \begin{smallmatrix}
          0\\
          -\beta_{n - 1}
        \end{smallmatrix}
      \right]} \Sigma A_1\oplus B_n \xrightarrow{\left[
        \begin{smallmatrix}
          -1 & 0\\
          (-1)^{n + 1} \Sigma \varphi_1 & (-1)^{n + 1} \beta_n
        \end{smallmatrix}
      \right]} \Sigma A_1\oplus \Sigma B_1,
  \end{align*}
  respectively.  In order to apply Theorem \ref{thm:octa} we need to
  prove that these $n$-$\Sigma$-sequences are $n$-angles.

  It can easily be shown that $X_\bullet$ is isomorphic to the direct
  sum of the trivial $n$-angle on $B_2$ and the left rotations of the
  trivial $n$-angles on $A_2$ and $B_1$.  Next, the
  $n$-$\Sigma$-sequence $Y_\bullet$ is isomorphic to the direct sum of
  the $n$-angle $A_\bullet$, the trivial $n$-angle on $B_1$ and the
  right rotation of the trivial $n$-angle on $B_2$.  Similarly, the
  $n$-$\Sigma$-sequence $Z_\bullet$ is isomorphic to the direct sum of
  the $n$-angle $B_\bullet$ and the left rotation of the trivial
  $n$-angle on $A_1$.  Hence, by (N1)(a) it follows that $X_\bullet$,
  $Y_\bullet$ and $Z_\bullet$ are $n$-angles.

  Since $X_\bullet$, $Y_\bullet$ and $Z_\bullet$ are $n$-angles, we
  may apply axiom (N4*) to the above diagram. Consequently, there
  exist morphisms $\sigma_3,\sigma_4,\ldots,\sigma_n$ and a morphism
  $\psi_{2n - 5}$ with the following three properties:
  \begin{enumerate}
  \item the sequence $\left(1, \left[
        \begin{smallmatrix}
          1 & -\varphi_2 & -\beta_1
        \end{smallmatrix}
      \right],\sigma_3,\sigma_4,\ldots,\sigma_n\right)$ is a morphism  
    $Y_\bullet \to Z_\bullet$ of $n$-angles,
  \item $\psi_{2n - 5}$ is a morphism $\Sigma A_1 \oplus B_n \to \Sigma A_2
    \oplus \Sigma B_1$ with
    \begin{equation*}
      \left[
        \begin{smallmatrix}
          (-1)^n \Sigma \varphi_2 & (-1)^n \Sigma \beta_1\\
          (-1)^n & 0\\
          0 & (-1)^n
        \end{smallmatrix}
      \right] \circ \psi_{2n - 5} = \left[
        \begin{smallmatrix}
          0 & \hfill 0\\
          (-1)^n \Sigma \alpha_1 & \hfill 0\\
          0 & -1
        \end{smallmatrix}
      \right] \circ \left[
        \begin{smallmatrix}
          -1 & 0\\
          (-1)^{n + 1} \Sigma \varphi_1 & (-1)^{n + 1} \beta_n
        \end{smallmatrix}
      \right].
    \end{equation*}
  \item the $n$-$\Sigma$-sequence
    \begin{equation*}
      A_3 \oplus B_2 \xrightarrow{\left[
          \begin{smallmatrix}
            -\alpha_3 & 0\\
            \hfill \sigma_{3,1} & \sigma_{3,2}
          \end{smallmatrix}
        \right]} A_4 \oplus B_3 \xrightarrow{\left[
          \begin{smallmatrix}
            -\alpha_4 & 0\\
            \hfill \sigma_4 & \beta_3
          \end{smallmatrix}
        \right]} \cdots % \xrightarrow{\left[
        %   \begin{smallmatrix}
        %     (-1)^n \sigma_{n,1} & 0\\
        %     (-1)^n \sigma_{n,2} & \beta_{n - 1}
        %   \end{smallmatrix}
        % \right]} \Sigma A_1 \oplus  B_n
      \xrightarrow{\psi_{2n - 5}} \Sigma A_2 \oplus \Sigma B_1
      \xrightarrow{\left[
          \begin{smallmatrix}
            (-1)^{n + 1} \Sigma \alpha_2 & 0\\
            (-1)^n \Sigma \varphi_2 & (-1)^n \Sigma \beta_1
          \end{smallmatrix}
        \right]} \Sigma A_3 \oplus \Sigma B_2
    \end{equation*}
    is an $n$-angle.
  \end{enumerate}
  Observe that the $n$-angle $X_\bullet$ consists of the zero object
  at positions $3$ through $n - 1$, thus the other morphisms
  $\psi_1,\psi_2,\ldots,\psi_{2n - 6}$ given by (N4*) are all zero.

  From property (1) the diagram
  \begin{center}
    \begin{tikzpicture}
      % Top row
      \node (A1B1oppe) at (0,4.5){$A_1\oplus B_1$};
      \node (B2A2B1) at (3,4.5){$B_2\oplus A_2\oplus B_1$};
      \node (A3B2) at (5.75,4.5){$A_3\oplus B_2$};
      \node (A4) at (7.75,4.5){$A_4$};
      \node (Adotsoppe) at (9.25,4.5){$\cdots$};
      
      % Top row (continued)
      \node (Adotsnede) at (4,1.5){};
      \node (An-1) at (5.5,1.5){$A_{n - 1}$};
      \node (An) at (7.75,1.5){$A_n$};
      \node (SA1SB1oppe) at (12,1.5){$\Sigma A_1\oplus \Sigma B_1$};
      
      % Bottom row
      \node (A1B1nede) at (0,3){$A_1\oplus B_1$};
      \node (B2) at (3,3){$B_2$};
      \node (B3) at (5.75,3){$B_3$};
      \node (B4) at (7.75,3){$B_4$};
      \node (Bdotsoppe) at (9.25,3){$\cdots$};
      
      % Bottom row (continued)
      \node (Bdotsnede) at (4,0){};
      \node (Bn-1) at (5.5,0){$B_{n - 1}$};
      \node (SA1Bn) at (7.75,0){$\Sigma A_1\oplus B_n$};
      \node (SA1SB1nede) at (12,0){$\Sigma A_1\oplus \Sigma B_1$};
      
      \begin{scope}[font=\scriptsize,->,midway]
        % Horizontal arrows
        \begin{scope}[above]
          \draw (A1B1oppe) -- node{$\left[
              \begin{smallmatrix}
                0 & \hfill 0\\
                (-1)^n\alpha_1 & \hfill 0\\
                0 & -1
              \end{smallmatrix}
            \right]$} (B2A2B1);
          \draw (B2A2B1) -- node{$\left[
              \begin{smallmatrix}
                0 & -\alpha_2 & 0\\
                1 & 0 & 0
              \end{smallmatrix}
            \right]$} (A3B2);
          \draw (A3B2) -- node{$\left[
              \begin{smallmatrix}
                -\alpha_3 & 0
              \end{smallmatrix}
            \right]$} (A4);
          \draw (A4) -- node{$\alpha_4$} (Adotsoppe);
          \draw (Adotsnede) -- node{$\alpha_{n - 2}$} (An-1);
          \draw (An-1) -- node{$\alpha_{n - 1}$} (An);
          \draw (An) -- node{$\left[
              \begin{smallmatrix}
                (-1)^n\alpha_n\\
                0
              \end{smallmatrix}
            \right]$} (SA1SB1oppe);
          \draw (A1B1nede) -- node{$\left[
              \begin{smallmatrix}
                (-1)^{n + 1} \varphi_2 \circ \alpha_1 & \beta_1
              \end{smallmatrix}
            \right]$} (B2);
          \draw (B2) -- node{$\beta_2$} (B3);
          \draw (B3) -- node{$-\beta_3$} (B4);
          \draw (B4) -- node{$-\beta_4$} (Bdotsoppe);
          \draw (Bdotsnede) -- node{$-\beta_{n - 2}$} (Bn-1);
          \draw (Bn-1) -- node{$\left[
              \begin{smallmatrix}
                0\\
                -\beta_{n - 1}
              \end{smallmatrix}
            \right]$} (SA1Bn);
          \draw (SA1Bn) -- node{$\left[
              \begin{smallmatrix}
                -1 & 0\\
                (-1)^{n + 1} \Sigma \varphi_1 & (-1)^{n + 1} \beta_n
              \end{smallmatrix}
            \right]$} (SA1SB1nede);
        \end{scope}
        
        % Vertical arrows
        \begin{scope}[right]
          \draw[-,double equal sign distance] (A1B1oppe) --
            (A1B1nede);
          \draw (B2A2B1) -- node{$\left[
              \begin{smallmatrix}
                1 & -\varphi_2 & -\beta_1
              \end{smallmatrix}
            \right]$} (B2);
          \draw (A3B2) -- node{$\sigma_3$} (B3);
          \draw (A4) -- node{$\sigma_4$} (B4);
          \draw (An-1) -- node{$\sigma_{n - 1}$} (Bn-1);
          \draw (An) -- node{$\sigma_n$} (SA1Bn);
          \draw[-,double equal sign distance] (SA1SB1oppe) --
            (SA1SB1nede);
        \end{scope}
      \end{scope}
    \end{tikzpicture}
  \end{center}
  is commutative.  Using the commuatativity, we can conclude that
  \begin{equation*}
    \begin{array}{r@{\ }c@{\ }l}
      \sigma_3 & = & \left[
        \begin{smallmatrix}
          \varphi_3 & \beta_2
        \end{smallmatrix}
      \right],\\
      \sigma_4 & = & \varphi_4,\\
      \sigma_5 & = & - \varphi_5 \\
      & \vdots &\\
      \sigma_{n - 1} & = & (-1)^{n - 1} \varphi_{n - 1},\\
      \sigma_n & = & \left[
        \begin{smallmatrix}
          (-1)^{n + 1} \alpha_n\\
          (-1)^n \varphi_n
        \end{smallmatrix}
      \right]
    \end{array}
  \end{equation*}
  for some morphisms $A_i \xrightarrow{\varphi_i} B_i$ ($3 \le i \le
  n$) making the sequence $\varphi =
  (\varphi_1,\varphi_2,\ldots,\varphi_n)$ into a morphism $A_\bullet
  \xrightarrow{\varphi} B_\bullet$ of $n$-angles.

  Next, consider the morphism $\psi_{2n - 5}$. Using property (2), we
  see that
  \begin{align*}
    \left[
      \begin{smallmatrix}
        (-1)^n \Sigma \varphi_2 & (-1)^n \Sigma \beta_1\\
        (-1)^n & 0\\
        0 & (-1)^n
      \end{smallmatrix}
    \right] \circ \psi_{2n - 5} &= \left[
      \begin{smallmatrix}
        0 & \hfill 0\\
        (-1)^n \Sigma \alpha_1 & \hfill 0\\
        0 & -1
      \end{smallmatrix}
    \right] \circ \left[
      \begin{smallmatrix}
        -1 & 0\\
        (-1)^{n + 1} \Sigma \varphi_1 & (-1)^{n + 1} \beta_n
      \end{smallmatrix}
    \right]\\
    &= \left[
      \begin{smallmatrix}
        0 & 0\\
        (-1)^{n + 1} \Sigma \alpha_1 & 0\\
        (-1)^n \Sigma \varphi_1 & (-1)^n \beta_n
      \end{smallmatrix}
      \right].
  \end{align*}
  Thus the morphism $\psi_{2n - 5}$ is given by the matrix
  \begin{equation*}
    \psi_{2n - 5} = \left[
      \begin{smallmatrix}
        -\Sigma \alpha_1 & 0\\
        \hfill \Sigma \varphi_1 & \beta_n
      \end{smallmatrix}
    \right].
  \end{equation*}
  Finally, from property (3) and what we have shown so far, the
  $n$-$\Sigma$-sequence
  \begin{equation*}
    A_3 \oplus B_2 \xrightarrow{\left[
        \begin{smallmatrix}
          -\alpha_3 & 0\\
          \hfill \varphi_3 & \beta_2
        \end{smallmatrix}
      \right]} A_4 \oplus B_3 \xrightarrow{\left[
        \begin{smallmatrix}
          -\alpha_4 & 0\\
          \hfill \varphi_4 & \beta_3
        \end{smallmatrix}
      \right]} \cdots \xrightarrow{\left[
        \begin{smallmatrix} 
          -\Sigma \alpha_1 & 0\\
          \hfill \Sigma \varphi_1 & \beta_n
        \end{smallmatrix}
      \right]} \Sigma A_2 \oplus \Sigma B_1 \xrightarrow{\left[
        \begin{smallmatrix}
          (-1)^{n + 1} \Sigma \alpha_2 & 0\\
          (-1)^n \Sigma \varphi_2 & (-1)^n \Sigma \beta_1
        \end{smallmatrix}
      \right]} \Sigma A_3 \oplus \Sigma B_2
  \end{equation*}
  is an $n$-angle.  Its right rotation
  \begin{equation*}
    A_2 \oplus B_1 \xrightarrow{\left[
        \begin{smallmatrix}
          -\alpha_2 & 0\\
          \hfill \varphi_2 & \beta_1
        \end{smallmatrix}
      \right]} A_3 \oplus B_2 \xrightarrow{\left[
        \begin{smallmatrix}
          -\alpha_3 & 0\\
          \hfill \varphi_3 & \beta_2
        \end{smallmatrix}
      \right]} \cdots \xrightarrow{\left[
        \begin{smallmatrix} 
          -\alpha_n & 0\\
          \hfill \varphi_n & \beta_{n - 1}
        \end{smallmatrix}
      \right]} \Sigma A_1 \oplus B_n \xrightarrow{\left[
        \begin{smallmatrix}
          -\Sigma \alpha_1 & 0\\
          \hfill \Sigma \varphi_1 & \beta_n
        \end{smallmatrix}
      \right]} \Sigma A_2 \oplus \Sigma B_1
  \end{equation*}
  is the mapping cone of $\varphi$, and this is an $n$-angle by axiom
  (N2). This completes the proof.
\end{proof}

Collecting Theorem \ref{thm:octa} and Theorem \ref{thm:opposite} gives
the following.

\begin{theorem}
  \label{thm:N4N4*}
  If $\nang$ is a collection of $n$-$\Sigma$-sequences satisfying
  axioms \emph{(N1),(N2)} and \emph{(N3)}, then the following are
  equivalent:
  \begin{enumerate}
  \item $\nang$ satisfies \emph{(N4)},
  \item $\nang$ satisfies \emph{(N4*)}.
  \end{enumerate}
\end{theorem}

We now discuss the case when $n = 3$, that is, when our category $\C$
is a triangulated category.  In this case, the classical octahedral
axiom, which was introduced by Verdier in \cite{Verdier1, Verdier2},
is the following:

\begin{itemize}
\item[\textbf{(TR4)}] Given a commutative diagram
  \begin{center}
    \begin{tikzpicture}
      % Top row      
      \node (A1oppe) at (0,1.5){$A_1$};
      \node (A2) at (2,1.5){$A_2$};
      \node (A3) at (4,1.5){$A_3$};
      \node (SA1oppe) at (6,1.5){$\Sigma A_1$};

      % Bottom row
      \node (A1nede) at (0,0){$A_1$};
      \node (B2) at (2,0){$B_2$};
      \node (B3) at (4,0){$B_3$};
      \node (SA1nede) at (6,0){$\Sigma A_1$};

      % Second column
      \node (C3) at (2,-1.5){$C_3$};
      \node (SA2) at (2,-3){$\Sigma A_2$};

      \begin{scope}[font=\scriptsize,->,midway]
        % Vertical arrows
        \draw[-,double equal sign distance] (A1oppe) -- (A1nede);
        \draw (A2) -- node[right]{$\varphi_2$} (B2);
        \draw[-,double equal sign distance] (SA1oppe) -- (SA1nede);
        \draw (B2) -- node[right]{$\gamma_2$} (C3);
        \draw (C3) -- node[right]{$\gamma_3$} (SA2); 
       
        % Horizontal arrows
        \draw (A1oppe) -- node[above]{$\alpha_1$} (A2);
        \draw (A2) -- node[above]{$\alpha_2$} (A3);
        \draw (A3) -- node[above]{$\alpha_3$} (SA1oppe);
        \draw (A1nede) -- node[above]{$\beta_1$}
          (B2);
        \draw (B2) -- node[above]{$\beta_2$} (B3);
        \draw (B3) -- node[above]{$\beta_3$} (SA1nede);
      \end{scope}
    \end{tikzpicture}
  \end{center}
  in which the top rows and second column are triangles.  Then there
  exist morphisms $A_3 \xrightarrow{\varphi_3} B_3$ and $B_3
  \xrightarrow{\psi_1} C_3$ with the following properties: the diagram
   \begin{center}
    \begin{tikzpicture}
      % Top row      
      \node (A1oppe) at (0,1.5){$A_1$};
      \node (A2) at (2,1.5){$A_2$};
      \node (A3) at (4,1.5){$A_3$};
      \node (SA1oppe) at (6,1.5){$\Sigma A_1$};

      % Bottom row
      \node (A1nede) at (0,0){$A_1$};
      \node (B2) at (2,0){$B_2$};
      \node (B3) at (4,0){$B_3$};
      \node (SA1nede) at (6,0){$\Sigma A_1$};

      % Second column
      \node (C3) at (2,-1.5){$C_3$};
      \node (SA2) at (2,-3){$\Sigma A_2$};
      
      % Third column
      \node (C3*) at (4,-1.5){$C_3$};
      \node (SA2*) at (4,-3){$\Sigma A_2$};
      
      % Square
      \node (S) at (3,-0.7){$\Theta$};

      \begin{scope}[font=\scriptsize,->,midway]
        % Vertical arrows
        \draw[-,double equal sign distance] (A1oppe) -- (A1nede);
        \draw (A2) -- node[right]{$\varphi_2$} (B2);
        \draw (A3) -- node[right]{$\varphi_3$} (B3);
        \draw[-,double equal sign distance] (SA1oppe) -- (SA1nede);
        \draw (B2) -- node[right]{$\gamma_2$} (C3);
        \draw (B3) -- node[right]{$\psi_1$} (C3*);
        \draw (C3) -- node[right]{$\gamma_3$} (SA2); 
        \draw (C3*) -- node[right]{$\Sigma \alpha_2 \circ
          \gamma_3$} (SA2*);
       
        % Horizontal arrows
        \draw (A1oppe) -- node[above]{$\alpha_1$} (A2);
        \draw (A2) -- node[above]{$\alpha_2$} (A3);
        \draw (A3) -- node[above]{$\alpha_3$} (SA1oppe);
        \draw (A1nede) -- node[above]{$\beta_1$}
          (B2);
        \draw (B2) -- node[above]{$\beta_2$} (B3);
        \draw (B3) -- node[above]{$\beta_3$} (SA1nede);
        \draw[-,double equal sign distance] (C3) -- (C3*);
        \draw (SA2) -- node[above]{$\Sigma \alpha_2$} (SA2*);
      \end{scope}
    \end{tikzpicture}
  \end{center}
  is commutative, the third column is a triangle, and $\gamma_3 \circ
  \psi_1 = \Sigma \alpha_1 \circ \beta_3$.
\end{itemize}

This is almost the same as our axiom (N4*): there is one difference.
Namely, axiom (N4*) does \emph{not} guarantee that the square $\Theta$
commutes.  However, when $n = 3$ and we start with the diagram given
in (TR4), then in the proof of Theorem \ref{thm:octa} we obtain the
commutative diagram
\begin{center}
  \begin{tikzpicture}
    % Top row
    \node (A2oppe) at (0,1.5){$A_2$};
    \node (A3B2) at (2,1.5){$A_3 \oplus B_2$};
    \node (B3) at (4,1.5){$B_3$};
    \node (SA2oppe) at (6,1.5){$\Sigma A_2$};
    
    % Bottom row
    \node (A2nede) at (0,0){$A_2$};
    \node (B2) at (2,0){$B_2$};
    \node (C3) at (4,0){$C_3$};
    \node (SA2nede) at (6,0){$\Sigma A_2$};
    
    % Vertical arrows
    \begin{scope}[font=\scriptsize,->,midway,right]
      \draw[-,double equal sign distance] (A2oppe) -- (A2nede);
      \draw (A3B2) -- node{$\left[
          \begin{smallmatrix}
            0 & 1
          \end{smallmatrix}
        \right]$} (B2);
      \draw (B3) -- node{$\psi_{1}$} (C3);
      \draw[-,double equal sign distance] (SA2oppe) -- (SA2nede);
    \end{scope}
    
    % Horizontal arrows
    \begin{scope}[font=\scriptsize,->,midway,above]
      \draw (A2oppe) -- node{$\left[
          \begin{smallmatrix}
            -\alpha_2\\
            \hfill \varphi_2
          \end{smallmatrix}
        \right]$} (A3B2);
      \draw (A3B2) -- node{$\left[
          \begin{smallmatrix}
            \varphi_3 & \beta_{2}
          \end{smallmatrix}
        \right]$} (B3);
      \draw (B3) -- node{$\Sigma \alpha_1 \circ \beta_3$} (SA2oppe);
      \draw (A2nede) -- node{$\varphi_2$} (B2);
      \draw (B2) -- node{$\gamma_2$} (C3);
      \draw (C3) -- node{$\gamma_3$} (SA2nede);
    \end{scope}
  \end{tikzpicture}
\end{center}

The commutativity of the middle square implies that the square
$\Theta$ in (TR4) commutes. Therefore, we recover the original
octahedral axiom (TR4) from axioms (N1), (N2), (N3) and
(N4). Conversely, Neeman proves in \cite[Theorem 1.8]{N1} that axioms
(N1), (N2), (N3) and (TR4) together imply axiom (N4). Consequently,
when $n = 3$ and the collection $\nang$ of $3$-$\Sigma$-sequences
satisfies axioms (N1),(N2) and (N3), then the following are
equivalent:
\begin{itemize} 
\item[(1)] $\nang$ satisfies (N4),
\item[(2)] $\nang$ satisfies (TR4),
\item[(3)] $\nang$ satisfies (N4*).
\end{itemize} 

We end this section with a discussion of homotopy cartesian
diagrams. Recall that when $n = 3$, then a commutative square
\begin{center}
  \begin{tikzpicture}
    % Top row
    \node (A1) at (0,1.5){$A_1$};
    \node (A2) at (2,1.5){$A_2$};
    
    % Bottom row
    \node (B1) at (0,0){$B_1$};
    \node (B2) at (2,0){$B_2$};

    % Vertical arrows
    \begin{scope}[font=\scriptsize,->,midway,right]
      \draw (A1) -- node{$\varphi_1$} (B1);
      \draw (A2) -- node{$\varphi_2$} (B2);
    \end{scope}
    
    % Horizontal arrows
    \begin{scope}[font=\scriptsize,->,midway,above]
      \draw (A1) -- node{$\alpha$} (A2);
      \draw (B1) -- node{$\beta$} (B2);
    \end{scope}
  \end{tikzpicture}
\end{center}
is \emph{homotopy cartesian} if there exists a triangle
\begin{equation*}
  A_1 \xrightarrow{\left[
      \begin{smallmatrix}
        -\alpha\phantom{_1}\\
        \phantom{-} \varphi_1
      \end{smallmatrix}
    \right]} A_2 \oplus B_1 
  \xrightarrow{\left[
      \begin{smallmatrix}
        \varphi_2 & \beta
      \end{smallmatrix}
    \right]} B_2
  \xrightarrow{\partial} \Sigma A_1
\end{equation*}
for some morphism $B_2 \xrightarrow{\partial} \Sigma A_1$.  Now let
(TR4*) be the axiom which is the same as (TR4), but with the
additional requirement that the commutative square
\begin{center}
  \begin{tikzpicture}
    % Top row
    \node (A2) at (0,1.5){$A_2$};
    \node (A3) at (2,1.5){$A_3$};
    
    % Bottom row
    \node (B2) at (0,0){$B_2$};
    \node (B3) at (2,0){$B_3$};

    % Vertical arrows
    \begin{scope}[font=\scriptsize,->,midway,right]
      \draw (A2) -- node{$\varphi_2$} (B2);
      \draw (A3) -- node{$\varphi_3$} (B3);
    \end{scope}
    
    % Horizontal arrows
    \begin{scope}[font=\scriptsize,->,midway,above]
      \draw (A2) -- node{$\alpha_2$} (A3);
      \draw (B2) -- node{$\beta_2$} (B3);
    \end{scope}
  \end{tikzpicture}
\end{center}
is homotopy cartesian.  Neeman shows in \cite{N1,N2} that (TR4) is
equivalent to the stronger (TR4*).  Consequently, the axioms (N4),
(N4*), (TR4) and (TR4*) are all equivalent.

Now let $\C$ be $n$-angulated.  Motivated by the above, we say that a
commutative diagram
\begin{center}
  \begin{tikzpicture}
    % Top row
    \node (A1) at (0,1.5){$A_1$};
    \node (A2) at (2,1.5){$A_2$};
    \node (Upperdots) at (4,1.5){$\cdots$};
    \node (An-2) at (6,1.5){$A_{n-2}$};
    \node (An-1) at (8,1.5){$A_{n-1}$};
    
    % Bottom row
    \node (B1) at (0,0){$B_1$};
    \node (B2) at (2,0){$B_2$};
    \node (Lowerdots) at (4,0){$\cdots$};
    \node (Bn-2) at (6,0){$B_{n-2}$};
    \node (Bn-1) at (8,0){$B_{n-1}$};

    % Vertical arrows
    \begin{scope}[font=\scriptsize,->,midway,right]
      \draw (A1) -- node{$\varphi_1$} (B1);
      \draw (A2) -- node{$\varphi_2$} (B2);
      \draw (An-2) -- node{$\varphi_{n-2}$} (Bn-2);
      \draw (An-1) -- node{$\varphi_{n-1}$} (Bn-1);
    \end{scope}
    
    % Horizontal arrows
    \begin{scope}[font=\scriptsize,->,midway,above]
      \draw (A1) -- node{$\alpha_1$} (A2);
      \draw (A2) -- node{$\alpha_2$} (Upperdots);
      \draw (Upperdots) -- node{$\alpha_{n-3}$} (An-2);
      \draw (An-2) -- node{$\alpha_{n-2}$} (An-1);
      \draw (B1) -- node{$\beta_1$} (B2);
      \draw (B2) -- node{$\beta_2$} (Lowerdots);
      \draw (Lowerdots) -- node{$\beta_{n-3}$} (Bn-2);
      \draw (Bn-2) -- node{$\beta_{n-2}$} (Bn-1);
    \end{scope}
  \end{tikzpicture}
\end{center}
is \emph{homotopy cartesian} if the $n$-$\Sigma$-sequence
\begin{equation*}
  A_1 \xrightarrow{\left[
      \begin{smallmatrix}
        -\alpha_1\\
        \hfill \varphi_1
      \end{smallmatrix} \right]} A_2 \oplus B_1 \xrightarrow{\left[
      \begin{smallmatrix}
        \alpha_2 & 0\\
        \varphi_2 & \beta_1
      \end{smallmatrix}
    \right]} A_3 \oplus B_2 \xrightarrow{\left[
      \begin{smallmatrix}
        \hfill \alpha_3 & 0\\
        -\varphi_3 & \beta_2
      \end{smallmatrix}
    \right]} \dots \xrightarrow{\left[
      \begin{smallmatrix}
        \alpha_{n - 2} & 0\\
        (-1)^n \varphi_{n - 2} & \beta_{n - 3}
      \end{smallmatrix}
    \right]} A_{n-1} \oplus B_{n - 2} \xrightarrow{\left[
      \begin{smallmatrix}
        (-1)^{n + 1} \varphi_{n-1} & \beta_{n - 2}
      \end{smallmatrix}
    \right]} B_{n-1} \xrightarrow{\partial} \Sigma A_1
\end{equation*}
is an $n$-angle for some morphism $B_{n-1} \xrightarrow{\partial}
\Sigma A_1$.  In the proof of Theorem \ref{thm:octa}, when we showed
that axiom (N4*) follows from axiom (N4), we proved in addition that
the commutative diagram
\begin{center}
  \begin{tikzpicture}
    % Top row
    \node (A1) at (0,1.5){$A_2$};
    \node (A2) at (2,1.5){$A_3$};
    \node (Upperdots) at (4,1.5){$\cdots$};
    \node (An-2) at (6,1.5){$A_{n-1}$};
    \node (An-1) at (8,1.5){$A_{n}$};
    
    % Bottom row
    \node (B1) at (0,0){$B_2$};
    \node (B2) at (2,0){$B_3$};
    \node (Lowerdots) at (4,0){$\cdots$};
    \node (Bn-2) at (6,0){$B_{n-1}$};
    \node (Bn-1) at (8,0){$B_{n}$};

    % Vertical arrows
    \begin{scope}[font=\scriptsize,->,midway,right]
      \draw (A1) -- node{$\varphi_2$} (B1);
      \draw (A2) -- node{$\varphi_3$} (B2);
      \draw (An-2) -- node{$\varphi_{n-1}$} (Bn-2);
      \draw (An-1) -- node{$\varphi_{n}$} (Bn-1);
    \end{scope}
    
    % Horizontal arrows
    \begin{scope}[font=\scriptsize,->,midway,above]
      \draw (A1) -- node{$\alpha_2$} (A2);
      \draw (A2) -- node{$\alpha_3$} (Upperdots);
      \draw (Upperdots) -- node{$\alpha_{n-2}$} (An-2);
      \draw (An-2) -- node{$\alpha_{n-1}$} (An-1);
      \draw (B1) -- node{$\beta_2$} (B2);
      \draw (B2) -- node{$\beta_3$} (Lowerdots);
      \draw (Lowerdots) -- node{$\beta_{n-2}$} (Bn-2);
      \draw (Bn-2) -- node{$\beta_{n-1}$} (Bn-1);
    \end{scope}
  \end{tikzpicture}
\end{center} 
is homotopy cartesian.  In fact, that was precisely Lemma
\ref{lem:hmtpy}.  Consequently, axiom (N4*) (and then also axiom (N4))
is equivalent to the stronger axiom which requires the above
commutative diagram to be homotopy cartesian.

\section{An example: the Geiss-Keller-Oppermann-category}
\label{sec:example}
In this section, we recall the example from \cite{GKO} of an
$n$-angulated category arising from an $(n-2)$-cluster tilting
subcategory of a triangulated category.  Let $\T$ be a triangulated
category with suspension $\Sigma$, and let $\C$ be a full subcategory.
Then $\C$ is an \emph{$(n-2)$-cluster tilting} subcategory if it
satisfies the following two properties:
\begin{enumerate}
\item $\C$ is functorially finite in $\T$.
\item $\C$ is given by
\begin{align*}
  \C & = \{X \in \T \mid \Hom_{\T}(X, \Sigma^iC) = 0 \text{ for } 1
  \leq i \leq n-3 \text{ and } \forall C \in \C\}\\
  & = \{X \in \T \mid \Hom_{\T}(\Sigma^iC,X) = 0 \text{ for } 1 \leq i
  \leq n - 3 \text{ and } \forall C \in \C\}.
\end{align*}
\end{enumerate}
Suppose in addition that $\C$ is closed under the automorphism
$\Sigma^{n-2}$ of $\T$, and denote $\Sigma^{n-2}$ by $\susp$. Let
$\nang$ be the collection of $n$-$\susp$-sequences
\begin{equation*}
  A_1 \xrightarrow{\alpha_1} A_2 \xrightarrow{\alpha_2} \cdots
  \xrightarrow{\alpha_{n-1}} A_n \xrightarrow{\alpha_n} \susp A_1
\end{equation*}
in $\C$ such there exists a diagram 
\begin{center}
  \begin{tikzpicture}[text centered]
    % Top row
    \node (A2) at (1.5,1.5){$A_2$};
    \node (A3) at (4.5,1.5){$A_3$};
    \node (Adots) at (5.5,1.5){$\cdots$};
    \node (An-2) at (6.5,1.5){$A_{n-2}$};
    \node (An-1) at (9.5,1.5){$A_{n-1}$};
    
    % Bottom row
    \node (A1) at (0,0){$A_1$};
    \node (X1) at (3,0){$X_1$};
    \node (X2) at (6,0){$X_2$};
    \node (Xdots) at (7,0){$\cdots$};
    \node (Xn-3) at (8,0){$X_{n-3}$};
    \node (An) at (11,0){$A_n$};
    
    % Triangles
    \node (T1) at (1.5,0.8){$\Delta$};
    \node (T2) at (4.5,0.8){$\Delta$};
    \node (Tn-2) at (9.5,0.8){$\Delta$};
    
    \begin{scope}[->,font=\scriptsize,midway]
      % Vertical arrows
      \draw (A1) -- node[above left]{$\alpha_1$} (A2);
      \draw (A2) -- node[above right]{$f_1$} (X1);
      \draw (X1) -- node[above left]{$g_1$} (A3);
      \draw (A3) -- node[above right]{$f_2$} (X2);
      \draw (An-2) -- node[above right]{$f_{n-3}$} (Xn-3);
      \draw (Xn-3) -- node[above left]{$g_{n-3}$} (An-1);
      \draw (An-1) -- node[above right]{$\alpha_{n-1}$} (An);
      
      % Horizontal arrows
      \draw (A2) -- node[above]{$\alpha_2$} (A3);
      \draw (An-2) -- node[above]{$\alpha_{n-2}$} (An-1);
      \draw[|->] (X1) -- node[above]{$\partial_{n-2}$} (A1);
      \draw[|->] (X2) -- node[above]{$\partial_{n-3}$} (X1);
      \draw[|->] (An) -- node[above]{$\partial_1$} (Xn-3);     
    \end{scope}
  \end{tikzpicture}
\end{center}
in $\T$ with the following properties:
\begin{enumerate}
\item Each diagram triangle $\Delta$ is a triangle in $\T$, where a
  map $X \mapsto Y$ denotes a map from $X$ to $\Sigma Y$.
\item The other diagram triangles commute.
\item The map $\alpha_n$ equals the composition
  $\Sigma^{n-3} \partial_{n-2} \circ \Sigma^{n-4} \partial_{n-3} \circ
  \cdots \circ \partial_1$.
\end{enumerate}
Then it is shown in \cite[Section 3, Theorem 1]{GKO} that
$(\C,\susp,\nang)$ is an $n$-angulated category.  In what follows, we
verify the octahedral axiom (N4*) in the case when $n = 4$.  Thus we
assume that $\C$ is a $2$-cluster tilting subcategory of $\T$ closed
under the automorphism $\Sigma^2$, which we denote by $\susp$.

Suppose we are given a commutative diagram
\begin{center}
  \begin{tikzpicture}[text centered]
    % Top row
    \node (A1) at (0,1.5){$A_1$};
    \node (A2) at (1.5,1.5){$A_2$};
    \node (A3) at (3,1.5){$A_3$};
    \node (A4) at (4.5,1.5){$A_4$};
    \node (SA1) at (6,1.5){$\susp A_1$};
    
    % Bottom row
    \node (B1) at (0,0){$A_1$};
    \node (B2) at (1.5,0){$B_2$};
    \node (B3) at (3,0){$B_3$};
    \node (B4) at (4.5,0){$B_4$};
    \node (SB1) at (6,0){$\susp A_1$};
    
    % Second column
    \node (C3) at (1.5,-1.5){$C_3$};
    \node (C4) at (1.5,-3){$C_4$};
    \node (SA2) at (1.5,-4.5){$\susp A_2$};

    \begin{scope}[->,font=\scriptsize,midway]
      % Vertical arrows
      \draw[-,double equal sign distance] (A1) -- (B1);
      \draw (A2) -- node[right]{$\varphi_2$} (B2);
      \draw[-,double equal sign distance] (SA1) -- (SB1);
      \draw (B2) -- node[right]{$\gamma_2$} (C3);
      \draw (C3) -- node[right]{$\gamma_3$} (C4);
      \draw (C4) -- node[right]{$\gamma_4$} (SA2);
      
      % Horizontal arrows
      \draw (A1) -- node[above]{$\alpha_1$} (A2);
      \draw (A2) -- node[above]{$\alpha_2$} (A3);
      \draw (A3) -- node[above]{$\alpha_3$} (A4);
      \draw (A4) -- node[above]{$\alpha_4$} (SA1);
      
      \draw (B1) -- node[above]{$\beta_1$} (B2);
      \draw (B2) -- node[above]{$\beta_2$} (B3);
      \draw (B3) -- node[above]{$\beta_3$} (B4);
      \draw (B4) -- node[above]{$\beta_4$} (SB1);
     \end{scope}
  \end{tikzpicture}
\end{center}
in $\C$ whose top rows and second column are $4$-angles. We must find
morphisms $A_3 \xrightarrow{\varphi_3} B_3, A_4
\xrightarrow{\varphi_4} B_4$ and $\psi_1, \psi_2, \psi_3$ with the
following two properties:
\begin{enumerate}
\item The sequence $(1, \varphi_2, \varphi_3, \varphi_4 )$ is a
  morphism of $4$-angles.
\item The $4$-$\susp$-sequence
\begin{equation*}
  A_3 \xrightarrow{\left[
      \begin{smallmatrix}
        \alpha_3\\
        \varphi_3
      \end{smallmatrix}
    \right]} A_4 \oplus B_3 \xrightarrow{\left[
      \begin{smallmatrix}
        \varphi_4 & -\beta_3\\
        \psi_2 & \hfill \psi_1
      \end{smallmatrix}
    \right]} B_4 \oplus C_3 \xrightarrow{\left[
      \begin{smallmatrix}
        \psi_3 & \gamma_3\\
      \end{smallmatrix}
    \right]} C_4 \xrightarrow{\susp \alpha_2 \circ \gamma_4} \susp
  A_3
\end{equation*}
is a $4$-angle in $\C$, and $\gamma_4 \circ \psi_3 = \susp \alpha_1
\circ \beta_4$.
\end{enumerate}
We only use the axioms and properties of the underlying triangulated
category $\T$.  The three given $4$-angles correspond to three
diagrams
\begin{center}
  \begin{tikzpicture}[text centered]
    % Top row
    \node (A2) at (1.5,1.5){$A_2$};
    \node (A3) at (4.5,1.5){$A_3$};    
    \node (A1) at (0,0){$A_1$};
    \node (X) at (3,0){$X$};
    \node (A4) at (6,0){$A_4$};      
    \node (TA1) at (1.5,0.8){$\Delta_2$};
    \node (TA2) at (4.5,0.8){$\Delta_1$};
        
    \begin{scope}[->,font=\scriptsize,midway]
      \draw (A1) -- node[above left]{$\alpha_1$} (A2);
      \draw (A2) -- node[above right]{$f$} (X);
      \draw (X) -- node[above left]{$g$} (A3);
      \draw (A3) -- node[above right]{$\alpha_3$} (A4);
      \draw (A2) -- node[above]{$\alpha_2$} (A3);
      \draw[|->] (X) -- node[above]{$\partial_2$} (A1);
      \draw[|->] (A4) -- node[above]{$\partial_1$} (X);    
    \end{scope}
    
    % Second row
    \node (B2) at (1.5,-1){$B_2$};
    \node (B3) at (4.5,-1){$B_3$};    
    \node (B1) at (0,-2.5){$A_1$};
    \node (Y) at (3,-2.5){$Y$};
    \node (B4) at (6,-2.5){$B_4$};      
    \node (TB1) at (1.5,-1.7){$\Delta_2'$};
    \node (TB2) at (4.5,-1.7){$\Delta_1'$};
        
    \begin{scope}[->,font=\scriptsize,midway]
      \draw (B1) -- node[above left]{$\beta_1$} (B2);
      \draw (B2) -- node[above right]{$f'$} (Y);
      \draw (Y) -- node[above left]{$g'$} (B3);
      \draw (B3) -- node[above right]{$\beta_3$} (B4);
      \draw (B2) -- node[above]{$\beta_2$} (B3);
      \draw[|->] (Y) -- node[above]{$\partial'_2$} (B1);
      \draw[|->] (B4) -- node[above]{$\partial'_1$} (Y);    
    \end{scope}
    
    % Third row
    \node (C2) at (1.5,-3.5){$B_2$};
    \node (C3) at (4.5,-3.5){$C_3$};    
    \node (C1) at (0,-5){$A_2$};
    \node (Z) at (3,-5){$Z$};
    \node (C4) at (6,-5){$C_4$};      
    \node (TC1) at (1.5,-4.2){$\Delta_2''$};
    \node (TC2) at (4.5,-4.2){$\Delta_1''$};
        
    \begin{scope}[->,font=\scriptsize,midway]
      \draw (C1) -- node[above left]{$\varphi_2$} (C2);
      \draw (C2) -- node[above right]{$f''$} (Z);
      \draw (Z) -- node[above left]{$g''$} (C3);
      \draw (C3) -- node[above right]{$\gamma_3$} (C4);
      \draw (C2) -- node[above]{$\gamma_2$} (C3);
      \draw[|->] (Z) -- node[above]{$\partial''_2$} (C1);
      \draw[|->] (C4) -- node[above]{$\partial''_1$} (Z);    
    \end{scope}   
  \end{tikzpicture}
\end{center}
in $\T$, that is, each $4$-angle is built from two triangles.

\subsection*{Step 1}
Choose a morphism $X \xrightarrow{\varphi} Y$ in $\T$ such that the
mapping cone of the morphism
\begin{center}
  \begin{tikzpicture}[text centered]
    % Top row
    \node (A1) at (0,1.5){$A_1$};
    \node (A2) at (2,1.5){$A_2$};
    \node (X) at (4,1.5){$X$};
    \node (SA1) at (6,1.5){$\Sigma A_1$};
    
    % Bottom row
    \node (B1) at (0,0){$A_1$};
    \node (B2) at (2,0){$B_2$};
    \node (Y) at (4,0){$Y$};
    \node (SB1) at (6,0){$\Sigma A_1$};
    
    \begin{scope}[->,font=\scriptsize,midway]
      % Vertical arrows
      \draw[-,double equal sign distance] (A1) -- (B1);
      \draw (A2) -- node[right]{$\varphi_2$} (B2);
      \draw[->,dashed] (X) -- node[right]{$\varphi$} (Y);
      \draw[-,double equal sign distance] (SA1) -- (SB1);
      
      % Horizontal arrows
      \draw (A1) -- node[above]{$\alpha_1$} (A2);
      \draw (A2) -- node[above]{$f$} (X);
      \draw (X) -- node[above]{$\partial_2$} (SA1);
      
      \draw (B1) -- node[above]{$\beta_1$} (B2);
      \draw (B2) -- node[above]{$f'$} (Y);
      \draw (Y) -- node[above]{$\partial'_2$} (SB1);

     \end{scope}
  \end{tikzpicture}
\end{center}
is a triangle in $\T$:
\begin{equation*}
  A_2 \oplus A_1 \xrightarrow{\left[
      \begin{smallmatrix}
        -f\phantom{_2} & 0\\
        \hfill \varphi_2 & \beta_1
      \end{smallmatrix}
    \right]} X \oplus B_2 \xrightarrow{\left[
      \begin{smallmatrix}
        - \partial_2 & 0\\
        \varphi & \hfill f'
      \end{smallmatrix}
    \right]} \Sigma A_1 \oplus Y \xrightarrow{\left[
      \begin{smallmatrix}
        - \Sigma \alpha_1 & 0\\
        1 & \partial_2'
      \end{smallmatrix}
    \right]} 
  \Sigma A_2 \oplus \Sigma A_1.
\end{equation*}
The $3$-$\Sigma$ sequence
\begin{equation*}
  A_2 \xrightarrow{\left[
      \begin{smallmatrix}
        -f\phantom{_2}\\
        \hfill \varphi_2
      \end{smallmatrix}
    \right]} X \oplus B_2 \xrightarrow{\left[
      \begin{smallmatrix}
        \varphi & \hfill f'
      \end{smallmatrix}
    \right]} Y \xrightarrow{\Sigma \alpha_1 \circ \partial_2'} 
  \Sigma A_2 
\end{equation*}
is a direct summand, and therefore a triangle in $\T$.

\subsection*{Step 2}
By the $3 \times 3$ Lemma (cf.\ \cite[Lemma 2.6]{May}, there exists an
object $W$ in $\T$ and maps $w_1, \dots, w_5$ such that the diagram
\begin{center}
  \begin{tikzpicture}[text centered]
    % First row
    \node (X2) at (2,1.5){$\Sigma^{-1} A_4$};
    \node (X3) at (4,1.5){$\Sigma^{-1} A_4$};
    
    % Second row
    \node (Y1) at (0,0){$A_2$};
    \node (Y2) at (2,0){$X \oplus B_2$};
    \node (Y3) at (4,0){$Y$};
    \node (Y4) at (6,0){$\Sigma A_2$};
    
    % Third row
    \node (Z1) at (0,-1.5){$A_2$};
    \node (Z2) at (2,-1.5){$A_3 \oplus B_2$};
    \node (Z3) at (4,-1.5){$W$};
    \node (Z4) at (6,-1.5){$\Sigma A_2$};
    
    % Fourth row
    \node (U2) at (2,-3){$A_4$};
    \node (U3) at (4,-3){$A_4$};
    
    \begin{scope}[->,font=\scriptsize,midway]
      % Vertical arrows
      \draw[-,double equal sign distance] (Y1) -- (Z1);
      
      \draw (X2) -- node[right]{$\left[
      \begin{smallmatrix}
        -\Sigma^{-1} \partial_1 \\
        0
      \end{smallmatrix}
      \right]$} (Y2);
      \draw (Y2) -- node[right]{$\left[
      \begin{smallmatrix}
        g & 0 \\
        0 & 1
      \end{smallmatrix}
      \right]$} (Z2);
      \draw (Z2) -- node[right]{$\left[
      \begin{smallmatrix}
        \alpha_3 & 0
      \end{smallmatrix}
      \right]$} (U2);
      
      \draw (X3) -- node[right]{$-\varphi \circ
        \Sigma^{-1} \partial_1$} (Y3);
      \draw[->,dashed] (Y3) -- node[right]{$w_3$} (Z3);
      \draw[->,dashed] (Z3) -- node[right]{$w_5$} (U3);
      
      \draw[-,double equal sign distance] (Y4) -- (Z4);
            
      % Horizontal arrows
      \draw[-,double equal sign distance] (X2) -- (X3);
      
      \draw (Y1) -- node[above]{$\left[
          \begin{smallmatrix}
            -f\phantom{_2}\\
            \hfill \varphi_2
          \end{smallmatrix}
        \right]$} (Y2);
      \draw (Y2) -- node[above]{$\left[
          \begin{smallmatrix}
            \varphi & f' \\
          \end{smallmatrix}
        \right]$} (Y3);
      \draw (Y3) -- node[above]{$\Sigma \alpha_1 \circ \partial_2'$}
        (Y4);
      
      \draw (Z1) -- node[above]{$\left[
          \begin{smallmatrix}
            -\alpha_2 \\
            \hfill \varphi_2
          \end{smallmatrix}
        \right]$} (Z2);
      \draw[->,dashed] (Z2) -- node[above]{$\left[
          \begin{smallmatrix}
            w_1 & w_2 \\
          \end{smallmatrix}
        \right]$} (Z3);
      \draw[->,dashed] (Z3) -- node[above]{$w_4$} (Z4);

      \draw[-,double equal sign distance] (U2) -- (U3);
     \end{scope}
  \end{tikzpicture}
\end{center}
commutes, and all columns and rows are triangles.  The second row is
the triangle from step 1, whereas the second column is the direct sum
of the left rotation of $\Delta_1$ and the trivial triangle on $B_2$.
By the octahedral axiom, there exist maps $w_6,w_7$ such that the
diagram
\begin{center}
  \begin{tikzpicture}[text centered]
    % First row
    \node (X1) at (0,1.5){$\Sigma^{-1} A_4$};
    \node (X2) at (2,1.5){$Y$};
    \node (X3) at (4,1.5){$W$};
    \node (X4) at (6,1.5){$A_4$};
    
    % Second row
    \node (Y1) at (0,0){$\Sigma^{-1} A_4$};
    \node (Y2) at (2,0){$B_3$};
    \node (Y3) at (4,0){$A_4 \oplus B_3$};
    \node (Y4) at (6,0){$A_4$};
    
    % Third row
    \node (Z2) at (2,-1.5){$B_4$};
    \node (Z3) at (4,-1.5){$B_4$};
      
    % Fourth row
    \node (U2) at (2,-3){$\Sigma Y$};
    \node (U3) at (4,-3){$\Sigma W$};
    
    \begin{scope}[->,font=\scriptsize,midway]
      % Vertical arrows
      \draw[-,double equal sign distance] (X1) -- (Y1);
      
      \draw (X2) -- node[right]{$g'$} (Y2);
      \draw (Y2) -- node[right]{$\beta_3$} (Z2);
      \draw (Z2) -- node[right]{$\partial_1'$} (U2);
      
      \draw[->,dashed] (X3) -- node[right]{$\left[
          \begin{smallmatrix}
            w_5\\
            w_6
          \end{smallmatrix}
        \right]$} (Y3);
      \draw[->,dashed] (Y3) -- node[right]{$\left[
          \begin{smallmatrix}
            w_7 & \beta_3
          \end{smallmatrix}
        \right]$} (Z3);
      \draw (Z3) -- node[right]{$\Sigma w_3 \circ \partial_1'$} (U3);
      
      \draw[-,double equal sign distance] (X4) -- (Y4);
            
      % Horizontal arrows
      \draw (X1) -- node[above]{$-\varphi \circ
        \Sigma^{-1} \partial_1$} (X2);
      \draw (X2) -- node[above]{$w_3$} (X3);
      \draw (X3) -- node[above]{$w_5$} (X4);
      
      \draw (Y1) -- node[above]{$0$} (Y2);
      \draw (Y2) -- node[above]{$\left[
          \begin{smallmatrix}
            0 \\
            1
          \end{smallmatrix}
        \right]$} (Y3);
      \draw (Y3) -- node[above]{$\left[
          \begin{smallmatrix}
            1 & 0
          \end{smallmatrix}
        \right]$} (Y4);
      
      \draw[-,double equal sign distance] (Z2) -- (Z3);
      
      \draw (U2) -- node[above]{$\Sigma w_3$} (U3);
     \end{scope}
  \end{tikzpicture}
\end{center}
commutes, all rows and columns are triangles, and $\partial_1' \circ
w_7 = - \Sigma \varphi \circ \partial_1$.  The first row is a triangle
from the diagram above, the second row is a direct sum of trivial
triangles, and the second column is $\Delta_1'$.

\subsection*{Step 3}
Define maps
\begin{align*}
  A_3 & \xrightarrow{\varphi_3} B_3, \hspace{5mm} \varphi_3 \ceq w_6
  \circ w_1\\
  A_4 & \xrightarrow{\varphi_4} B_4, \hspace{5mm} \varphi_4 \ceq -
  w_7.
\end{align*}
>From the diagrams in steps 1 and 2 we obtain the equalities
\begin{align*}
  \varphi_3 \circ \alpha_2 &= w_6 \circ w_1 \circ \alpha_2 = w_6 \circ
  w_2 \circ \varphi_2 = w_6 \circ w_3 \circ f' \circ \varphi_2 = g'
  \circ f' \circ \varphi_2 = \beta_2 \circ \varphi_2\\
  \varphi_4 \circ \alpha_3 & = \varphi_4 \circ w_5 \circ w_1 = -w_7
  \circ w_5 \circ w_1 = \beta_3 \circ w_6 \circ w_1 = \beta_3 \circ
  \varphi_3\\
  \alpha_4 &= \Sigma \partial_2 \circ \partial_1 = \Sigma \partial_2'
  \circ \Sigma \varphi \circ \partial_1 = \Sigma \partial_2'
  \circ \partial_1' \circ (-w_7) = \beta_4 \circ \varphi_4
\end{align*}
of maps in $\T$.  This shows that the diagram
\begin{center}
  \begin{tikzpicture}[text centered]
    % Top row
    \node (X1) at (0,1.5){$A_1$};
    \node (X2) at (2,1.5){$A_2$};
    \node (X3) at (4,1.5){$A_3$};
    \node (X4) at (6,1.5){$A_4$};
    \node (X5) at (8,1.5){$\susp A_1$};
    
    % Bottom row
    \node (Y1) at (0,0){$A_1$};
    \node (Y2) at (2,0){$B_2$};
    \node (Y3) at (4,0){$B_3$};
    \node (Y4) at (6,0){$B_4$};
    \node (Y5) at (8,0){$\susp A_1$};
    
    \begin{scope}[->,font=\scriptsize,midway]
      % Vertical arrows
      \draw[-,double equal sign distance] (X1) -- (Y1);
      \draw (X2) -- node[right]{$\varphi_2$} (Y2);
      \draw (X3) -- node[right]{$\varphi_3$} (Y3);
      \draw (X4) -- node[right]{$\varphi_3$} (Y4);
      \draw[-,double equal sign distance] (X5) -- (Y5);
      
      % Horizontal arrows
      \draw (X1) -- node[above]{$\alpha_1$} (X2);
      \draw (X2) -- node[above]{$\alpha_2$} (X3);
      \draw (X3) -- node[above]{$\alpha_3$} (X4);
      \draw (X4) -- node[above]{$\alpha_4$} (X5);
      
      \draw (Y1) -- node[above]{$\beta_1$} (Y2);
      \draw (Y2) -- node[above]{$\beta_2$} (Y3);
      \draw (Y3) -- node[above]{$\beta_3$} (Y4);
      \draw (Y4) -- node[above]{$\beta_4$} (Y5);
     \end{scope}
  \end{tikzpicture}
\end{center}
is a morphism of $4$-angles in $\C$, hence the first part of axiom
(N4*) is satisfied.

\subsection*{Step 4}
Choose a morphism $W \xrightarrow{\psi} Z$ in $\T$ such that the
mapping cone of the morphism
\begin{center}
  \begin{tikzpicture}[text centered]
    % Top row
    \node (X1) at (0,1.5){$A_2$};
    \node (X2) at (2,1.5){$A_3 \oplus B_2$};
    \node (X3) at (4,1.5){$W$};
    \node (X4) at (6,1.5){$\Sigma A_2$};
    
    % Bottom row
    \node (Y1) at (0,0){$A_2$};
    \node (Y2) at (2,0){$B_2$};
    \node (Y3) at (4,0){$Z$};
    \node (Y4) at (6,0){$\Sigma A_2$};
    
    \begin{scope}[->,font=\scriptsize,midway]
      % Vertical arrows
      \draw[-,double equal sign distance] (X1) -- (Y1);
      \draw (X2) -- node[right]{$\left[
          \begin{smallmatrix}
            0 & 1
          \end{smallmatrix}
        \right]$} (Y2);
      \draw[->,dashed] (X3) -- node[right]{$\psi$} (Y3);
      \draw[-,double equal sign distance] (X4) -- (Y4);
      
      % Horizontal arrows
      \draw (X1) -- node[above]{$\left[
          \begin{smallmatrix}
            -\alpha_2 \\
            \hfill \varphi_2
          \end{smallmatrix}
        \right]$} (X2);
      \draw (X2) -- node[above]{$\left[
          \begin{smallmatrix}
            w_1 & w_2
          \end{smallmatrix}
        \right]$} (X3);
      \draw (X3) -- node[above]{$w_4$} (X4);
      
      \draw (Y1) -- node[above]{$\varphi_2$} (Y2);
      \draw (Y2) -- node[above]{$f''$} (Y3);
      \draw (Y3) -- node[above]{$\partial''_2$} (Y4);
     \end{scope}
  \end{tikzpicture}
\end{center}
is a triangle in $\T$:
\begin{equation*}
  A_3 \oplus B_2 \oplus A_2 \xrightarrow{\left[
      \begin{smallmatrix}
        -w_1 & -w_2 & 0\\
        0 & 1 & \varphi_2
      \end{smallmatrix}
    \right]} W \oplus B_2 \xrightarrow{\left[
      \begin{smallmatrix}
        -w_4 & 0\\
        \psi & f''
      \end{smallmatrix}
    \right]} \Sigma A_2 \oplus Z \xrightarrow{\left[
      \begin{smallmatrix}
        \hfill \Sigma \alpha_2 & 0\\
        -\Sigma \varphi_2 & 0\\
        1 & \partial_2''
      \end{smallmatrix}
    \right]} 
  \Sigma A_3 \oplus \Sigma B_2 \oplus \Sigma A_2.
\end{equation*}
The top row in the diagram is a triangle from step 2, whereas the
bottom row is $\Delta_2''$.  The $3$-$\Sigma$ sequence
\begin{equation*}
  A_3 \xrightarrow{-w_1} W \xrightarrow{\psi} Z \xrightarrow{- \Sigma
    \alpha_2 \circ \partial_2''} \Sigma A_3
\end{equation*}
is a direct summand of the mapping cone, and therefore a triangle in
$\T$.

\subsection*{Step 5}
By the $3 \times 3$ Lemma, there exists an object $U$ in $\T$ and maps
$u_1, \dots, u_5$ such that the diagram
\begin{center}
  \begin{tikzpicture}[text centered]
    % First row
    \node (X2) at (2,1.5){$\Sigma^{-1} B_4$};
    \node (X3) at (4,1.5){$\Sigma^{-1} B_4$};
    
    % Second row
    \node (Y1) at (0,0){$A_3$};
    \node (Y2) at (2,0){$W$};
    \node (Y3) at (4,0){$Z$};
    \node (Y4) at (6,0){$\Sigma A_3$};
    
    % Third row
    \node (Z1) at (0,-1.5){$A_3$};
    \node (Z2) at (2,-1.5){$A_4 \oplus B_3$};
    \node (Z3) at (4,-1.5){$U$};
    \node (Z4) at (6,-1.5){$\Sigma A_3$};
    
    % Fourth row
    \node (U2) at (2,-3){$B_4$};
    \node (U3) at (4,-3){$B_4$};
    
    \begin{scope}[->,font=\scriptsize,midway]
      % Vertical arrows
      \draw[-,double equal sign distance] (Y1) -- (Z1);
      
      \draw (X2) -- node[right]{$w_3 \circ \Sigma^{-1} \partial_1'$}
        (Y2);
      \draw (Y2) -- node[right]{$\left[
          \begin{smallmatrix}
            w_5\\
            w_6
          \end{smallmatrix}
        \right]$} (Z2);
      \draw (Z2) -- node[right]{$\left[
          \begin{smallmatrix}
            \varphi_4 & - \beta_3
          \end{smallmatrix}
        \right]$} (U2);
      
      \draw (X3) -- node[right]{$\psi \circ w_3 \circ
        \Sigma^{-1} \partial_1'$} (Y3);
      \draw[->,dashed] (Y3) -- node[right]{$u_3$} (Z3);
      \draw[->,dashed] (Z3) -- node[right]{$u_5$} (U3);
      
      \draw[-,double equal sign distance] (Y4) -- (Z4);
            
      % Horizontal arrows
      \draw[-,double equal sign distance] (X2) -- (X3);
      
      \draw (Y1) -- node[above]{$w_1$} (Y2);
      \draw (Y2) -- node[above]{$\psi$} (Y3);
      \draw (Y3) -- node[above]{$\Sigma \alpha_2 \circ \partial_2''$} (Y4);
      
      \draw (Z1) -- node[above]{$\left[
          \begin{smallmatrix}
            \alpha_3 \\
            \varphi_3
          \end{smallmatrix}
        \right]$} (Z2);
      \draw[->,dashed] (Z2) -- node[above]{$\left[
          \begin{smallmatrix}
            u_1 & u_2 \\
          \end{smallmatrix}
        \right]$} (Z3);
      \draw[->,dashed] (Z3) -- node[above]{$u_4$} (Z4);
      
      \draw[-,double equal sign distance] (U2) -- (U3);
     \end{scope}
  \end{tikzpicture}
\end{center}
commutes, and all columns and rows are triangles.  The second row is
isomorphic to the triangle from step 4, whereas the second column is
isomorphic to the right rotation of the third column in the octahedral
diagram in step 2.  By the octahedral axiom, there exist maps
$u_6,u_7$ such that the diagram
\begin{center}
  \begin{tikzpicture}[text centered]
    % First row
    \node (X1) at (0,1.5){$\Sigma^{-1} B_4$};
    \node (X2) at (2,1.5){$Z$};
    \node (X3) at (4,1.5){$U$};
    \node (X4) at (6,1.5){$B_4$};
    
    % Second row
    \node (Y1) at (0,0){$\Sigma^{-1} B_4$};
    \node (Y2) at (2,0){$C_3$};
    \node (Y3) at (4,0){$B_4 \oplus C_3$};
    \node (Y4) at (6,0){$B_4$};
    
    % Third row
    \node (Z2) at (2,-1.5){$C_4$};
    \node (Z3) at (4,-1.5){$C_4$};
      
    % Fourth row
    \node (U2) at (2,-3){$\Sigma Z$};
    \node (U3) at (4,-3){$\Sigma U$};
    
    \begin{scope}[->,font=\scriptsize,midway]
      % Vertical arrows
      \draw[-,double equal sign distance] (X1) -- (Y1);
      
      \draw (X2) -- node[right]{$g''$} (Y2);
      \draw (Y2) -- node[right]{$\gamma_3$} (Z2);
      \draw (Z2) -- node[right]{$\partial_1''$} (U2);
      
      \draw[->,dashed] (X3) -- node[right]{$\left[
          \begin{smallmatrix}
            u_5\\
            u_6
          \end{smallmatrix}
        \right]$} (Y3);
      \draw[->,dashed] (Y3) -- node[right]{$\left[
          \begin{smallmatrix}
            u_7 & \gamma_3
          \end{smallmatrix}
        \right]$} (Z3);
      \draw (Z3) -- node[right]{$\Sigma u_3 \circ \partial_1''$} (U3);
      
      \draw[-,double equal sign distance] (X4) -- (Y4);
            
      % Horizontal arrows
      \draw (X1) -- node[above]{$\psi \circ w_3 \circ
        \Sigma^{-1} \partial_1'$} (X2);
      \draw (X2) -- node[above]{$u_3$} (X3);
      \draw (X3) -- node[above]{$u_5$} (X4);
      
      \draw (Y1) -- node[above]{$0$} (Y2);
      \draw (Y2) -- node[above]{$\left[
          \begin{smallmatrix}
            0 \\
            1
          \end{smallmatrix}
        \right]$} (Y3);
      \draw (Y3) -- node[above]{$\left[
          \begin{smallmatrix}
            1 & 0
          \end{smallmatrix}
        \right]$} (Y4);
      
      \draw[-,double equal sign distance] (Z2) -- (Z3);
      
      \draw (U2) -- node[above]{$\Sigma u_3$} (U3);
     \end{scope}
  \end{tikzpicture}
\end{center}
commutes, all rows and columns are triangles, and $\partial_1'' \circ
u_7 = \Sigma \psi \circ \Sigma w_3 \circ \partial_1'$.  The first row
is a triangle from the diagram above, the second row is a direct sum
of trivial triangles, and the second column is $\Delta_1''$.

\subsection*{Step 6}
Define maps
\begin{align*}
  B_3 & \xrightarrow{\psi_1} C_3, \hspace{5mm} \psi_1 \ceq u_6 \circ
  u_2\\
  A_4 & \xrightarrow{\psi_2} C_3, \hspace{5mm} \psi_2 \ceq u_6 \circ
  u_1\\
  B_4 & \xrightarrow{\psi_3} C_4, \hspace{5mm} \psi_3 \ceq u_7,
\end{align*}
and consider the diagram
\begin{center}
  \begin{tikzpicture}[text centered]
    % Top row
    \node (A4B3) at (1,1.5){$A_4 \oplus B_3$};
    \node (B4C3) at (5,1.5){$B_4 \oplus C_3$};    
    \node (A3) at (-1,0){$A_3$};
    \node (U) at (3,0){$U$};
    \node (C4) at (7,0){$C_4$};      
    \node (T1) at (1,0.8){$\Delta_2'''$};
    \node (T2) at (5,0.8){$\Delta_1'''$};
    \node (T3) at (3,1.2){$\Lambda$};
        
    \begin{scope}[->,font=\scriptsize,midway]
      \draw (A3) -- node[above left]{$\left[
          \begin{smallmatrix}
            \alpha_3 \\
            \varphi_3
          \end{smallmatrix}
        \right]$} (A4B3);
      \draw (A4B3) -- node[above right]{$\left[
          \begin{smallmatrix}
            u_1 & u_2
          \end{smallmatrix}
        \right]$} (U);
      \draw (U) -- node[above left]{$\left[
          \begin{smallmatrix}
            u_5 \\
            u_6
          \end{smallmatrix}
        \right]$} (B4C3);
      \draw (B4C3) -- node[above right]{$\left[
          \begin{smallmatrix}
            \psi_3 & \gamma_3
          \end{smallmatrix}
        \right]$} (C4);
      \draw (A4B3) -- node[above]{$\left[
          \begin{smallmatrix}
            \varphi_4 & - \beta_3 \\
            \psi_2 & \psi_1
          \end{smallmatrix}
      \right]$} (B4C3);
      \draw[|->] (U) -- node[above]{$u_4$} (A3);
      \draw[|->] (C4) -- node[above]{$\Sigma u_3 \circ \partial_1''$}
        (U);
    \end{scope}  
  \end{tikzpicture}
\end{center}
in $\T$.  From the two diagrams in step 5, we know that $\Delta_1'''$
and $\Delta_2'''$ are triangles in $\T$.  Moreover, from the
commutativity of the bottom square of the top diagram in step 5, we
see that the triangle $\Lambda$ commutes.  Composition along the lower
edge gives
\begin{equation*}
  \Sigma u_4 \circ \Sigma u_3 \circ \partial_1'' =  \Sigma^2 \alpha_2
  \circ \Sigma \partial_2'' \circ \partial_1'' = \susp \alpha_2 \circ
  \gamma_4,
\end{equation*}
where we have used the commutativity of the rightmost square of the
top diagram in step 5.  By definition, the diagram therefore represents
a $4$-angle
\begin{equation*}
  A_3 \xrightarrow{\left[
      \begin{smallmatrix}
        \alpha_3\\
        \varphi_3
      \end{smallmatrix}
    \right]} A_4 \oplus B_3 \xrightarrow{\left[
      \begin{smallmatrix}
        \varphi_4 & -\beta_3\\
        \psi_2 & \hfill \psi_1
      \end{smallmatrix}
    \right]} B_4 \oplus C_3 \xrightarrow{\left[
      \begin{smallmatrix}
        \psi_3 & \gamma_3\\
      \end{smallmatrix}
    \right]} C_4 \xrightarrow{\susp \alpha_2 \circ \gamma_4} \susp
  A_3
\end{equation*}
in $\C$.  Finally, from the diagrams in steps 2,4 and 5 we obtain
\begin{equation*}
  \gamma_4 \circ \psi_3 = \Sigma \partial_2'' \circ \partial_1'' \circ
  u_7 = \Sigma \partial_2'' \circ \Sigma \psi \circ \Sigma w_3
  \circ \partial_1' = \Sigma w_4 \circ \Sigma w_3 \circ \partial_1' =
  \Sigma^2 \alpha_1 \circ \Sigma \partial_2' \circ \partial_1' = \susp
  \alpha_1 \circ \beta_4.
\end{equation*}
The conclusion here in step 6, together with that in step 3, shows
that the $4$-angulated category $\C$ satisfies the higher octahedral
axiom (N4*).

\section*{Acknowledgements}

We would like to thank Bernhard Keller, Henning Krause and Steffen Oppermann for valuable suggestions and comments.

\bibliographystyle{plain}

\begin{thebibliography}{10}
\bibitem{Balmer}P.\ Balmer, \emph{Separability and triangulated
    categories}, Adv.\ Math.\ \textbf{226} (2011), no.\ 5, 4352--4372.
\bibitem{GKO}C.\ Geiss, B.\ Keller and S.\ Oppermann,
  \emph{$n$-angulated categories}, to appear in J.\ Reine Angew.\ Math.
\bibitem{HJ}T.\ Holm and P.\ J{\o}rgensen, \emph{Triangulated
    categories: definitions, properties, and examples}, in
  Triangulated categories, 1--51, London Math.\ Soc.\ Lecture Note
  Ser., \textbf{375}, Cambridge Univ.\ Press, Cambridge, 2010.
\bibitem{KV}B.\ Keller and D.\ Vossieck, \emph{Sous les cat\'egories
    d\'eriv\'ees}, C.\ R.\ Acad.\ Sci.\ Paris \textbf{305} (1987),
  no.\ 6, 225--228.
\bibitem{May}J.P.\ May, \emph{The additivity of traces in triangulated
    categories}, Adv.\ Math.\ \textbf{163} (2001), no.\ 1, 34--73.
\bibitem{N1}A.\ Neeman, \emph{Some new axioms for triangulated
    categories}, J.\ Algebra \textbf{139} (1991), no.\ 1, 221--255.
\bibitem{N2}A.\ Neeman, \emph{Triangulated categories}, Annals of
  Mathematics Studies, \textbf{148}, Princeton University Press,
  Princeton, NJ, 2001, viii+449 pp.
\bibitem{Puppe}D.\ Puppe, \emph{On the structure of stable
    homotopy theory}, Colloquium on algebraic topology, Aarhus
  Universitet Matematisk Institutt, 1962, 65--71.
\bibitem{Verdier1}J.-L.\ Verdier, \emph{Cat{\'e}gories
    d{\'e}riv{\'e}es, {\'e}tat 0}, in Cohomologie {\'e}tale,
  S{\'e}minaire de G{\'e}om{\'e}trie Alg{\'e}brique du Bois-Marie SGA
  412, Avec la collaboration de J.\ F.\ Boutot, A.\ Grothendieck, L.\
  Illusie et J.-L.\ Verdier, Lecture Notes in Mathematics, Vol.\ 569,
  Springer-Verlag, Berlin-New York, 1977, iv+312pp.
\bibitem{Verdier2}J.-L.\ Verdier, \emph{Des cat{\'e}gories
    d{\'e}riv{\'e}es des cat{\'e}gories ab{\'e}liennes},
  Ast{\'e}risque No.\ 239 (1996), xii+253 pp.
\end{thebibliography}

\end{document}